\documentclass[12pt,a4paper]{amsart}

\usepackage{lscape}
\usepackage{amscd}
\usepackage{amssymb}
\usepackage{amsmath}
\usepackage{enumitem}
\usepackage[pdftex]{graphicx}
\usepackage{tikz}
\usepackage{tikz-3dplot}
\usepackage{verbatim}

\def\dr{{\partial_r}}

\def\SO{\mathrm{SO}}

\def\Sp{\mathrm{Sp}}

\def\SU{\mathrm{SU}}

\def\Sp{\mathrm{Sp}}
\def\Spin{\mathrm{Spin}}

\def\G{\mathrm{G}}
\def\U{\mathrm{U}}
\def\T{\mathrm{T}}

\def\Hol{{\rm Hol}}
\def\f{\varphi}

\def \RM{\mathbb{R}}

\def \ZM{\mathbb{Z}}
\def \CM{\mathbb{C}}
\def \HM{\mathbb{H}}
\def \SM{\mathbb{S}}

\def\G{\mathrm{G}}

\def\1{\mathbf{1}}
\def\#{\sharp}

\def\C{\mathbb{C}}

\def\dd{\mathrm d}
\def\e{\varepsilon}

\def\cL{\mathcal{L}}
\def\cV{\mathcal{V}}

\def\Ric{\mathrm{Ric}}

\def\tr{\mathrm{tr}}

\def\<#1,#2>{\langle\,#1,\,#2\,\rangle}

\def\beq{\begin{equation}}
\def\eeq{\end{equation}}

\def\norm(#1){\|#1\|^2}
\def\rectangle(#1,#2)[#3,#4]#5{
 \multiput(#1,#2)(#3,0)2{\line(0,1){#4}}\multiput(#1,#2)(0,#4)2{\line(1,0){#3}}
 \put(#1,#2){\vbox to #4pt{\hbox to #3pt{\hfill}\vfill}}}
\def\recttext(#1,#2)[#3,#4]#5{\put(#1,#2)
 {\vbox to #4pt{\vfill\hbox to #3pt{\hss#5\hss}\vfill}}}
\newtheorem{Lemma}{Lemma}[section]
\newtheorem{Proposition}[Lemma]{Proposition}

\newtheorem{Theorem}[Lemma]{Theorem}
\newtheorem{Corollary}[Lemma]{Corollary}

\theoremstyle{definition}

\newtheorem{Remark}[Lemma]{Remark}
\newtheorem{Example}[Lemma]{Example}

\title{Conformally related K\"ahler metrics and the holonomy of lcK manifolds}

\author{Farid Madani, Andrei Moroianu, Mihaela Pilca}
\thanks{This work was partially supported by the Procope Project No. 32977YJ. 
The third named author also acknowledges partial support 
from the ARS-Program at the University of Regensburg.}

\address{Farid Madani\\Institut f\"ur Mathematik\\
Goethe Universit\"at Frankfurt\\Robert-Mayer-Str. 10, 60325 Frankfurt am Main, Germany}
\email{madani@math.uni-frankfurt.de}

\address{Andrei Moroianu \\ Laboratoire de Math\'ematiques de Versailles, UVSQ, CNRS, Universit\'e Paris-Saclay, 78035 Versailles, France }
\email{andrei.moroianu@math.cnrs.fr}

\address{Mihaela Pilca\\Fakult\"at f\"ur Mathematik\\
Universit\"at Regensburg\\Universit\"atsstr. 31 
D-93040 Regensburg, Germany
\emph{and} 
Institute of Mathematics ``Simion Stoilow" of the Romanian Academy, 
21, Calea Grivitei Str.
010702-Bucharest, Romania}
\email{mihaela.pilca@mathematik.uni-regensburg.de}



\begin{document}

\begin{abstract} A locally conformally K\"ahler (lcK) manifold is a complex manifold $(M,J)$ together with a Hermitian metric $g$ which is conformal to a K\"ahler metric in the neighbourhood of each point.
In this paper we obtain three classification results in locally conformally K\"ahler geometry.
The first one is the classification of conformal classes on compact manifolds containing two non-homothetic 
K\"ahler metrics. The second one is the classification of compact Einstein locally conformally K\"ahler manifolds. The third result is 
the classification of the possible (restricted) Riemannian holonomy groups of compact locally conformally K\"ahler manifolds. We show that every locally (but not globally) conformally K\"ahler compact manifold of dimension $2n$ has holonomy $\SO(2n)$, unless it is Vaisman, in which case it has restricted holonomy $\SO(2n-1)$. We also show that the restricted holonomy of a proper globally conformally K\"ahler compact manifold of dimension $2n$ is either $\SO(2n)$, or $\SO(2n-1)$, or $\U(n)$, and we give the complete description of the possible solutions in the last two cases.
\end{abstract}
\maketitle

\section{Introduction}
It is well-known that on a compact complex manifold, any conformal class admits at most one K\"ahler metric compatible with the complex structure, 
up to a positive constant. The situation might change if the complex structure is not fixed.
One may thus naturally ask the following question: are there any compact manifolds which admit 
two non-homothetic metrics in the same conformal class, which are both K\"ahler 
(then necessarily with respect to non-conjugate complex structures)? One of the aims of the present paper 
is to answer this question by describing all such manifolds. This problem can be interpreted in terms of 
conformally K\"ahler metrics in real dimension $2n$ with Riemannian holonomy contained in the unitary group $\U(n)$. More generally, 
we want to classify {\em locally} conformally K\"ahler metrics on compact manifolds which are Einstein or have non-generic holonomy.

Recall that a Hermitian manifold $(M,g,J)$ of complex dimension $n\geq 2$ is called locally conformally K\"ahler (lcK) if around every point in $M$ the metric $g$ can be conformally rescaled to a K\"ahler metric. If $\Omega:=g(J\cdot,\cdot)$ denotes the fundamental 2-form, the above condition is equivalent to the existence of a closed 1-form $\theta$, called the Lee form (which is up to a constant equal to the logarithmic differential of the local conformal factors), such that 
\begin{equation*}\label{eqomega}\dd\Omega=2\theta\wedge \Omega.\end{equation*}

If the Lee form $\theta$ vanishes, the structure $(g,J)$ is simply K\"ahler. 
If the Lee form does not vanish identically, the lcK structure is called {\em proper}.
When $\theta$ is exact, there exists a K\"ahler metric in the conformal class of $g$, 
and the manifold is called {\em globally conformally K\"ahler} (gcK). 
If $\theta$ is not exact, then $(M,g,J)$ it is called {\em strictly lcK}.
A particular class of proper lcK manifolds is the one consisting 
of manifolds whose Lee form is parallel with respect to the Levi-Civita 
connexion of the metric, called Vaisman manifolds. A Vaisman manifold is always strictly lcK since the Lee form, being harmonic, cannot be exact.  
 
In this paper we study three apparently independent -- but actually interrelated -- classification problems:

{\bf P1.} The classification of compact proper lcK manifolds $(M^{2n},g,J,\theta)$ with $g$ Einstein.

{\bf P2.}  The classification of compact conformal manifolds $(M^{2n},c)$ whose conformal class $c$ contains two non-homothetic K\"ahler metrics.

{\bf P3.}  The classification of compact proper lcK manifolds $(M^{2n},g,J,\theta)$ with reduced ({\em i.e.} non-generic) holonomy: $\Hol(M,g)\subsetneq\SO(2n)$.

It turns out that P1 and P2 are important steps (but also interesting for their own sake) towards the solution of P3. 

We are able to solve each of these problems completely. Their solutions are provided by Theorem \ref{p1}, Theorem \ref{p2} and Theorem \ref{p3} below. We now explain briefly these results and describe the methods used to prove them.

In complex dimension $2$, C.~LeBrun~\cite{LeBrun1997} showed that if a 
compact complex surface admits an Einstein metric compatible with the complex structure, then the metric is gcK and the complex surface is obtained from $\CM P^2$ by blowing up one, two or three points in general position. When the complex dimension is greater than $2$,
A. Derdzinski and G. Maschler, \cite{DM2003},  have obtained the local classification of 
conformally-Einstein K\"ahler metrics, and showed that in the compact case the 
only K\"ahler metrics which are conformal (but not homothetic) to an Einstein metric are those constructed by L. B\'erard-Bergery in \cite{Bergery1982}. 
By changing the point of view, this can be interpreted as the classification of compact (proper) globally conformally K\"ahler manifolds $(M^{2n},g,J,\theta)$ with $g$ Einstein. In order to solve P1, it remains to understand the strictly lcK case. 

Since every strictly lcK manifold has infinite fundamental group, Myers' theorem shows that the scalar curvature of any compact Einstein strictly lcK manifold is non-positive. In Theorem~\ref{thm einstein} below we show, using Weitzenb\"ock-type arguments, that the Lee form of every compact Einstein lcK manifold with non-positive scalar curvature vanishes. This gives the solution to Problem P1:

\begin{Theorem}\label{p1} If $(g,J,\theta)$ is an Einstein proper lcK structure on a compact manifold $M^{2n}$, then the Lee form is exact ($\theta=\dd\f$), and the scalar curvature of $g$ is positive. For $n=2$ the complex surface $(M,J)$ is obtained from $\CM P^2$ by blowing up one, two or three points in general position. For $n\ge3$, the K\"ahler manifold $(M,e^{-2\f}g,J)$ is one of the examples of conformally-Einstein K\"ahler manifolds constructed by B\'erard-Bergery in \cite{Bergery1982}.
\end{Theorem}

The solution of Problem P2 relies on Theorem~\ref{thm kaehler} below, where we show that if $(M^{2n},g,J,\theta)$ is a compact proper lcK manifold whose metric $g$ is K\"ahler with respect to some complex structure $I$, then $I$ commutes with $J$ and the Lee form is exact: $\theta=\dd\f$. In particular $(g,I)$ and $(e^{-2\f}g,J)$ are K\"ahler structures on $M$, {\em i.e.} the conformal structure $[g]$ is ambik\"ahler, according to the terminology introduced in dimension 4 by V.~Apostolov, D.~Calderbank and P.~Gauduchon in \cite{ACG2013}. 

Examples of ambik\"ahler structures in every complex dimension $n\ge 2$ can be obtained on the total spaces of some $S^2$-bundles over compact Hodge manifolds, by an Ansatz which is reminiscent of Calabi's construction \cite{Calabi1982}. This construction is described in Proposition \ref{prop} below. 

Conversely, we have the following result, which answers Problem P2:

\begin{Theorem} \label{p2} Assume that a conformal class on a compact manifold $M$ of real dimension $2n\ge4$ contains two non-homothetic K\"ahler metrics $g_+$ and $g_-$, that is, there exist complex structures $J_+$ and $J_-$ and a non-constant function $\f$ such that $(g_+,J_+)$ and $(g_-:=e^{-2\f}g_+,J_-)$ are K\"ahler structures. Then $J_+$ and $J_-$ commute, so that $M$ is ambik\"ahler. Moreover, for $n\ge 3$,
there exists a compact K\"ahler manifold $(N,h,J_N)$, a positive real number $b$, and a function $\ell:(0,b)\to\RM^{>0}$ such that $(M,g_+,J_+)$ and $(M,g_-,J_-)$ are obtained from the construction described in Proposition \ref{prop}.
\end{Theorem}

The proof, explained in detail in Sections \ref{s3} and \ref{skahler}, goes roughly as follows: the main difficulty is to show that the complex structures $J_+$ and $J_-$ necessarily commute. This is done in Theorem \ref{thm kaehler} using in an essential way the compactness assumption. When the complex dimension is at least 3, Theorem \ref{thm kaehler} also shows that $\dd\f$ is preserved (up to sign) by $J_+J_-$.
As a consequence, one can check that $J_++J_-$ defines a Hamiltonian 2-form of rank 1 with respect to both K\"ahler metrics $g_+$ and $g_-$. One can then either use the classification of compact manifolds with Hamiltonian forms obtained in \cite{ACG2006} (which however is rather involved) or obtain the result in a simpler way by a geometric argument given in Proposition \ref{thmconv}.

We now discuss the holonomy problem for compact proper lcK manifolds, that is, Problem P3, whose original motivation stems from~\cite{Moroianu2015}. 

By the Berger-Simons holonomy theorem, an lcK manifold $(M^{2n},g,J)$ either has reducible restricted holonomy representation, or is locally symmetric irreducible, or its restricted holonomy group $\Hol_0(M,g)$ is one of the following: $\SO(2n)$, $\U(n)$, $\SU(n)$, $\Sp(\frac{n}{2})$, $\Sp(\frac{n}{2})\Sp(1)$, 
$\Spin(7)$. 

The restricted holonomy representation of a compact Riemannian manifold $(M,g)$ is reducible if and only if the tangent bundle of a finite covering of $M$ carries an oriented parallel (proper) distribution. We first show in Theorem~\ref{thmred} that a compact proper lcK manifold $(M^{2n},g,J)$ cannot carry a parallel distribution whose rank $d$ satisfies $2\le d\le 2n-2$. The special case when this distribution is 1-dimensional was studied recently in \cite{Moroianu2015}, where the second named author described all compact proper lcK manifolds $(M^{2n},g,J)$ with $n\ge3$ which carry a non-trivial parallel vector field. In Theorem \ref{main am} below we give an alternate proof of this classification, which is not only simpler, but also covers the missing case $n=2$. This settles the reducible case.

The remaining possible cases given by the Berger-Simons theorem are either Einstein or K\"ahler (and gcK by  Theorem~\ref{thm kaehler}), and thus fall into the previous classification results. Summarizing, we have the following classification result for the possible (restricted) holonomy groups of compact proper lcK manifolds:

\begin{Theorem}\label{p3}
Let $(M^{2n},g,J,\theta)$, $n\ge 2$, be a compact proper lcK manifold with non-generic holonomy group $\Hol(M,g)\subsetneq\SO(2n)$. Then the following exclusive possibilities occur:
\begin{enumerate}[label=\arabic*.]
\item \label{item 1} $(M,g,J,\theta)$ is strictly lcK, $\Hol(M,g)\simeq\SO(2n-1)$ and $(M,g,J,\theta)$ is Vaisman (that is, $\theta$ is parallel).
\item \label{item 2} $(M,g,J,\theta)$ is gcK (that is, $\theta$ is exact) and either:
\begin{enumerate}[label=\alph*),leftmargin=0.5cm]
\item $n\ge 3$, $\Hol_0(M,g)\simeq\U(n)$, and a finite covering of $(M,g,J,\theta)$ is obtained by the Calabi Ansatz described in Proposition \ref{prop}, or
\item $n=2$, $\Hol_0(M,g)\simeq\U(2)$ and $M$ is ambik\"ahler in the sense of \cite{ACG2013}, or
\item $\Hol_0(M,g)\simeq\SO(2n-1)$ and a finite covering of $(M,g,J,\theta)$ is obtained by the construction described in Theorem \ref{main am}.
\end{enumerate}
\end{enumerate}
\end{Theorem}

\section{Preliminaries on lcK manifolds}

A locally conformally K\"ahler (lcK) manifold is a connected Hermitian manifold $(M,g,J)$ of real dimension $2n\ge 4$ such that around each point, $g$ is conformal to a metric which is K\"ahler with respect to $J$. The covariant derivative of $J$ with respect to the Levi-Civita connection $\nabla$ of $g$ is determined by a closed 1-form $\theta$ (called the Lee form) via the formula (see e.g. \cite{Moroianu2015}):
\begin{equation}\label{nablaJ}
\nabla_X J=X\wedge J\theta+JX\wedge \theta, \qquad\forall\ X\in \T M.
\end{equation}
Recall that if $\tau$ is any $1$-form on $M$, $J\tau$ 
is the $1$-form defined by $(J\tau)(X):=-\tau(JX)$ 
for every $X\in\T M$, and $X\wedge\tau$ denotes 
the endomorphism of $\T M$ defined by $(X\wedge\tau)(Y):=g(X,Y)\tau^\sharp-\tau(Y)X$.
We will often identify $1$-forms and vector fields via the metric $g$, which will also be denoted by $\langle\cdot,\cdot\rangle$ when there is no ambiguity.

Let $\Omega:=g(J\cdot,\cdot)$ denote the associated 2-form of $J$. 
By \eqref{nablaJ}, its exterior derivative and co-differential are given by 
\begin{equation}\label{dOmega}\dd\Omega=2\theta\wedge\Omega,\end{equation}
and 
\begin{equation}\label{delta Omega}\delta\Omega=(2-2n)J\theta.\end{equation}

If $\theta\equiv0$, the structure $(g,J)$ is simply K\"ahler. If $\theta$ is not identically zero, then the lcK structure $(g,J,\theta)$ is called {\em proper}.
If $\theta=\dd\f$ is exact, then $\dd(e^{-2\f}\Omega)=0$, so
the conformally modified structure $(e^{-2\f}g,J)$ is K\"ahler, and
the structure $(g,J,\theta)$ is called {\em globally conformally K\"ahler} (gcK). 
The lcK structure is called {\em strictly} lcK if the Lee form $\theta$ is not exact and {\em Vaisman} if $\theta$ is parallel with respect to the Levi-Civita connexion of $g$.

A typical example of strictly lcK manifold, which is actually Vaisman, is $\SM^1\times \SM^{2n-1}$, 
endowed with the complex structure induced by the diffeomorphism 
$$(\C^n\setminus\{0\})/ \mathbb Z\longrightarrow \SM^1\times \SM^{2n-1}, \;\;
[z]\longmapsto \left(e^{2\pi i\ln |z|}, \frac{z}{|z|}\right),$$
 where $[z]:=\{e^k z\in \C^n\setminus\{0\}\,\vert\, k\in\mathbb Z\}$. The Lee form of this lcK structure is 
the length element of $\SM^1$, which is parallel.   

\begin{Remark}\label{cov}
For each lcK manifold $(M,g,J,\theta)$ there exists a group homomorphism 
from  $\pi_1(M)$ to $(\RM,+)$ which is trivial if and only if the structure is gcK. 
Indeed, $\pi_1(M)$ acts on the universal covering $\widetilde M$ of $M$, 
and preserves the induced lcK structure $(\tilde g, \tilde J, \tilde \theta)$. 
Since $\tilde\theta=\dd \f$ is exact on $\widetilde M$, for every 
$\gamma\in\pi_1(M)$ we have $\dd(\gamma^*\f)=\gamma^*(\dd \f)=\gamma^*(\tilde\theta)=\tilde\theta=\dd \f$, so there exists some real number $c_\gamma$ such that $\gamma^* \f=\f+c_\gamma$. The map $\gamma\mapsto c_\gamma$ is clearly a group morphism from  $\pi_1(M)$ to $(\RM,+)$, which is trivial if and only if $\theta$ is exact on $M$. This shows, in particular, that if $\pi_1(M)$ is finite, then every lcK structure on $M$ is gcK. 
\end{Remark}

For later use, we express, for every lcK structure $(g,J,\theta)$, the action of the Riemannian curvature tensor of $g$ on the Hermitian structure $J$.
\begin{Lemma}
The following formula holds for every vector fields $X,Y$ on a lcK manifold $(M,g,J,\theta)$:
\begin{multline}\label{RJ}
R_{X,Y}J=\theta(X)Y\wedge J\theta -\theta(Y)X\wedge J\theta
-\theta(Y)JX\wedge \theta +\theta(X)JY\wedge \theta\\
-|\theta|^2Y\wedge JX+|\theta|^2X\wedge JY+Y\wedge J\nabla_{X} \theta+JY\wedge \nabla_{X} \theta -X\wedge J\nabla_{Y}\theta 
 -JX\wedge \nabla_{Y}\theta.
\end{multline}
\end{Lemma}
\begin{proof}
Taking $X,Y$ parallel at the point where the computation is done and  applying \eqref{nablaJ}, 
we obtain:
\begin{eqnarray*}
R_{X,Y} J&=&\nabla_{X}(Y\wedge J\theta+JY\wedge \theta)-\nabla_{Y}(X\wedge J\theta+
JX\wedge \theta)\\&=&
Y\wedge (\nabla_XJ)(\theta)+(\nabla_XJ)(Y)\wedge \theta-X\wedge (\nabla_YJ)(\theta)-
(\nabla_YJ)(X)\wedge \theta\\&&
+Y\wedge J\nabla_{X} \theta+JY\wedge \nabla_{X} \theta -X\wedge J\nabla_{Y}\theta 
 -JX\wedge \nabla_{Y}\theta,
\end{eqnarray*}
which gives  \eqref{RJ} after a straightforward calculation using \eqref{nablaJ} again.
\end{proof}

Let  $\{e_i\}_{i=1,\ldots,2n}$ be a local orthonormal basis of $\T M$.
Substituting  $Y=e_j$ in \eqref{RJ}, taking the interior product with $e_j$ and summing over $j=1,\ldots,2n$ yields:
\begin{equation}\label{RJcontr}
\sum_{j=1}^{2n} (R_{X,e_j} J )(e_j)=(2n-3)\left(\theta(X) J\theta-|\theta|^2JX+ J\nabla_{X} \theta\right)-\theta(JX) \theta    -\nabla_{JX}\theta-JX\delta\theta,
\end{equation}
since the sum $\sum_{j=1}^{2n}g(J\nabla_{e_j}\theta,e_j)$ vanishes, as $\nabla\theta$ is symmetric.

\begin{Corollary}\label{flat}
If the metric $g$ of a compact lcK manifold $(M,g,J,\theta)$ is flat, then $\theta\equiv0$.
\end{Corollary}
\begin{proof}
If the Riemannian curvature of $g$ vanishes, \eqref{RJcontr} yields
\begin{equation*}
0=(2n-3)\left(\theta(X) J\theta-|\theta|^2JX+ J\nabla_{X} \theta\right)-\theta(JX) \theta    -\nabla_{JX}\theta-JX\delta\theta.
\end{equation*}

We make the scalar product with $JX$ in this equation for $X=e_j$, where $\{e_j\}_{j=1,\ldots,2n}$ is a local orthonormal basis of $\T M$, and sum over $j=1,\ldots,2n$ to obtain:
\begin{eqnarray*}
0&=&(2n-3)\left(|\theta|^2-2n|\theta|^2-\delta \theta\right)-|\theta|^2 +\delta\theta-2n\delta\theta\\
&=&-(2n-2)^2|\theta|^2-2(2n-2)\delta\theta.
\end{eqnarray*}
Since $n\ge 2$, this last equation yields $\delta\theta=(1-n)|\theta|^2$, which by Stokes' Theorem after integration over $M$ gives $\theta\equiv0$.
\end{proof}

The following example shows that the corollary does not hold without the compactness assumption.

\begin{Example}\label{expl}
Consider the standard flat K\"ahler structure $(g_0,J_0)$ on $M:=\mathbb{C}^n\setminus \{0\}$. If $r$ denotes the map $x\mapsto r(x):=|x|$, the conformal metric $g:=r^{-4}g_0$ on $M$ is gcK with respect to $J_0$, with Lee form $\theta=-2\dd\ln r$. Moreover $g$ is flat, being the pull-back of $g_0$ through the inversion $x\mapsto x/r^2$.
\end{Example}

\section{Compact Einstein lcK manifolds}\label{s2}
The purpose of this section is to classify compact Einstein proper lcK manifolds. 
We treat separately the two possible cases: non-negative and positive scalar curvature. In the non-negative case, we show that the 
Lee form must vanish, so the manifold is already K\"ahler. In the positive case, it follows that the manifold is globally conformally K\"ahler and one can use Maschler-Derdzinski's classification of conformally-Einstein K\"ahler metrics, for complex dimension $n\geq 3$, and the results of X. Chen, C. LeBrun and B. Weber \cite{CLW2008} for complex surfaces. 

Let $(M,g,J,\theta)$ be an lcK manifold. We denote by $S$ the following symmetric 2-tensor:
\begin{equation}\label{defS}
S:=\nabla \theta+\theta \otimes\theta,
\end{equation}
identified with a symmetric endomorphism via the metric $g$. 
\begin{Lemma}\label{commute}
On an lcK manifold $(M,g,J,\theta)$ with $g$ Einstein, $S$ commutes with $J$.
\end{Lemma}

\begin{proof}
Since the statement is local, we may assume without loss of generality that the Lee form is exact, $\theta=\dd \f$, which means that $g_K:=e^{-2\f}g$ is K\"ahler with respect to $J$. We denote the Einstein constant of $g$ by $\lambda$. 

The formula relating the Ricci tensors of conformally equivalent metrics \cite[Theorem 1.159]{Besse2008} reads:
\begin{equation*}
\Ric^K=\Ric^g+2(n-1)\left(\nabla^g \dd \f +\dd \f \otimes \dd \f \right)-\left(\Delta^g \f +2(n-1)g(\dd \f ,\dd \f )\right)g.
\end{equation*}
Since $g^K$ is K\"ahler, $\Ric^K(J\cdot,J\cdot)=\Ric^K(\cdot,\cdot)$. Using this fact, together with $g(J\cdot,J\cdot)=g(\cdot,\cdot)$ and $\Ric^g=\lambda g$ in the above formula, we infer:
\begin{equation}\label{com}
\left(\nabla^g \dd \f +\dd \f \otimes \dd \f \right)(J\cdot,J\cdot)=\left(\nabla^g \dd \f +\dd \f \otimes \dd \f \right)(\cdot,\cdot),
\end{equation}
which is equivalent to $SJ=JS$.
\end{proof}

The main result of this section is the following:

\begin{Theorem}\label{thm einstein}
If $(M,g,J,\theta)$ is a compact lcK manifold and $g$ is Einstein with non-positive scalar curvature, then $\theta\equiv0$, so $(M,g,J)$ is a K\"ahler-Einstein manifold.
\end{Theorem}

\begin{proof}
Let $\{e_i\}_{i=1,\dots, 2n}$ be a local orthonormal basis which is
parallel at the point where the computation is done.
We denote by $\lambda\leq 0$ the Einstein constant of the metric $g$, so $\Ric=\lambda g$.
The strategy of the proof is to apply the Bochner formula to the $1$-forms $\theta$ and $J\theta$ in order to obtain a formula relating the Einstein constant, the co-differential of the Lee form and its square norm, which leads to a contradiction (if $\theta$ is not identically zero) at a point where $\delta\theta+|\theta|^2$ attains its maximum.

Let $S$ denote as above the endomorphism $S=\nabla \theta+\theta \otimes\theta$. In particular, we have
\begin{equation}\label{Sth}
S\theta=\nabla_{\theta} \theta+|\theta|^2\theta=\frac{1}{2}\dd |\theta|^2+|\theta|^2\theta
\end{equation}
and the trace of $S$ is computed as follows
\begin{equation}\label{trS}
\tr(S)=|\theta|^2-\delta\theta.
\end{equation}
In the sequel, we use Lemma~\ref{commute}, ensuring that $S$ commutes with $J$. We start by computing the covariant derivative of $J\theta$:
\begin{equation}\label{eq nablaJth}
\begin{split}
\nabla_X J\theta&=(\nabla_X J)(\theta)+ J(\nabla_X \theta)\overset{\eqref{nablaJ}}{=}(X\wedge J\theta+JX\wedge\theta)(\theta)+J(SX-\theta(X)\theta)\\
&=JS X-J\theta(X)\theta-|\theta|^2JX.
\end{split}
\end{equation}
The exterior differential of $J\theta$ is then given by the following formula: 
\begin{equation}\label{diffJth}
\begin{split}
\dd J\theta=\sum_{i=1}^{2n} e_i\wedge \nabla_{e_i}J\theta=2JS +\theta\wedge J\theta-2|\theta|^2 \Omega.
\end{split}
\end{equation}
We further compute the Lie bracket between $\theta$ and $J\theta$ (viewed as vector fields):
\begin{equation}\label{lieJth} 
 [\theta, J\theta]=\nabla_{\theta} J\theta-\nabla_{J\theta}\theta\overset{\eqref{defS}, \eqref{eq nablaJth}}{=} JS\theta-|\theta|^2J\theta-SJ\theta=-|\theta|^2J\theta.
 \end{equation} 

By \eqref{delta Omega}, we have $\delta J\theta=0$. Using the following identities:
\begin{equation}\label{codiffth}
\delta(\theta\wedge J\theta) =(\delta\theta)J\theta -\delta(J\theta)\theta-[\theta,J\theta]\overset{\eqref{lieJth}}{=}(\delta\theta+|\theta|^2) J\theta,
\end{equation}
\begin{equation}\label{codiffom}
\delta(|\theta|^2\Omega) =-J(\dd|\theta|^2 )+|\theta|^2\delta\Omega \overset{\eqref{delta Omega}}{=}-J(\dd |\theta|^2)+(2-2n)|\theta|^2J\theta,
\end{equation}
we compute the Laplacian of $J\theta$: 
\begin{equation}\label{eq lap}
\begin{split}
\Delta J\theta&=\delta \dd J\theta \overset{\eqref{diffJth}}{=}\delta(2JS +\theta\wedge J\theta-2|\theta|^2 \Omega)\\
&\overset{\eqref{codiffth}, \eqref{codiffom}}{=}2\delta(JS)+ (\delta\theta+|\theta|^2) J\theta+2J(\dd |\theta|^2)+2(2n-2)|\theta|^2J\theta\\
&=2\delta(JS)+\delta\theta J\theta+2J(\dd |\theta|^2)+(4n-3)|\theta|^2J\theta.
\end{split}
\end{equation}

We next compute the rough Laplacian of $J\theta$:
\begin{equation}\label{eq connlap}
\begin{split}
\nabla^*\nabla J\theta&=-\sum_{i=1}^{2n}\nabla_{e_i}\nabla_{e_i} J\theta\overset{\eqref{eq nablaJth}}{=}-\sum_{i=1}^{2n}\nabla_{e_i}( JS e_i-J\theta(e_i)\theta-|\theta|^2Je_i)\\
&=\delta (JS)+\nabla_{J\theta}\theta+J\dd (|\theta|^2)+|\theta|^2\sum_{i=1}^{2n}(\nabla_{e_i} J)(e_i)\\
&\overset{\eqref{delta Omega}}{=}\delta (JS)+SJ\theta+J\dd (|\theta|^2)+(2n-2)|\theta|^2J\theta\\
&\overset{\eqref{Sth}}{=}\delta (JS)+|\theta|^2J\theta+\frac{3}{2}J\dd (|\theta|^2)+(2n-2)|\theta|^2J\theta.
\end{split}
\end{equation}
Using the Bochner formula $\Delta J\theta=\nabla^*\nabla J\theta+\Ric(J\theta)$ together with \eqref{eq lap} and \eqref{eq connlap} yields:
\begin{equation*}
\delta(JS)=-(\delta\theta) J\theta-\frac{1}{2}J(\dd |\theta|^2)-2(n-1)|\theta|^2J\theta+\lambda J\theta,
\end{equation*}
which, after applying $J$ on both sides, reads:
\begin{equation}\label{eqJdel}
J\delta(JS)=(\delta\theta) \theta+\frac{1}{2}\dd |\theta|^2+2(n-1)|\theta|^2\theta-\lambda \theta.
\end{equation}

The rough Laplacian of $\theta$ is computed as follows:
\begin{equation}\label{connLth}
\begin{split}
\nabla^*\nabla\theta=-\sum_{i=1}^{2n}\nabla_{e_i}\nabla_{e_i} \theta=-\sum_{i=1}^{2n}\nabla_{e_i}(Se_i-\theta(e_i)\theta)=\delta S-(\delta\theta)\theta+\frac{1}{2}\dd |\theta|^2.
\end{split}
\end{equation}

The Bochner formula 
$\Delta\theta=\nabla^*\nabla\theta+\Ric(\theta)$, together with \eqref{connLth} yields
\begin{equation}\label{eqJdel2}
\delta S=(\delta\theta)\theta-\frac{1}{2}\dd |\theta|^2-\lambda \theta+\dd \delta\theta.
\end{equation}

On the other hand, we have:
\begin{equation*}
\begin{split}
\delta (JS)&=-\sum_{i=1}^{2n}(\nabla_{e_i} JS)(e_i)
=-\sum_{i=1}^{2n}(\nabla_{e_i} J)(S e_i)-\sum_{i=1}^{2n}J(\nabla_{e_i} S)(e_i)\\
&\overset{\eqref{nablaJ}}{=}-\sum_{i=1}^{2n}(e_i\wedge J\theta+Je_i\wedge\theta)(S e_i)+J(\delta S)\\
&=-\tr(S)J\theta+2JS\theta+J(\delta S)\overset{\eqref{Sth},\eqref{trS}}{=}(\delta\theta) J\theta+J(\dd |\theta|^2)+|\theta|^2J\theta+J(\delta S).
\end{split}
\end{equation*}
Applying $J$ to this equality yields
\begin{equation}\label{eqJdel3}
J\delta (JS)+\delta S=-(\delta\theta) \theta-\dd |\theta|^2-|\theta|^2\theta.
\end{equation}

Summing up \eqref{eqJdel} and \eqref{eqJdel2}, and comparing with \eqref{eqJdel3}, we obtain:
\begin{equation}\label{summ}
3(\delta\theta)\theta-2\lambda \theta+\dd \delta\theta+\dd |\theta|^2+(2n-1)|\theta|^2\theta=0.
\end{equation}
After introducing the function $f:=\delta\theta+|\theta|^2$, \eqref{summ} reads:
\begin{equation}\label{eqf}
\dd f=(2\lambda -3f+(4-2n)|\theta|^2)\theta.
\end{equation}

We argue by contradiction and assume that $\theta$ is not identically zero. In this case, the integral of $f$ over $M$ is positive. As $M$ is compact, there exists $p_0\in M$ at which $f$ attains its maximum, $f(p_0)>0$. In particular, we have $(\dd f)_{p_0}=0$ and $(\Delta f) (p_0)\geq 0$. Applying \eqref{eqf} at the point $p_0$ yields that $\theta_{p_0}=0$, because $2\lambda-3f(p_0)+(4-2n)|\theta_{p_0}|^2<0$. From the definition of $f$, it follows that $\delta\theta(p_0)>0$.

On the other hand, taking the co-differential of \eqref{eqf}, we obtain:
\begin{equation*}
\Delta f=(2\lambda -3f+(4-2n)|\theta|^2)\delta\theta +3\theta(f)+(2n-4)\theta(|\theta|^2).
\end{equation*}
Evaluating at $p_0$ leads to a contradiction, since the left-hand side is non-negative and the right-hand side is negative, as $\theta_{p_0}=0$ and $(2\lambda -3f(p_0))\delta\theta(p_0)<0$.  
Thus, $\theta\equiv0$.
\end{proof}

\medskip

Note that in complex dimension $n=2$, 
 C.~LeBrun~\cite{LeBrun1997} showed, by extending results of A.~Derdzinski~\cite{Derdzinski1983}, that a Hermitian non-K\"ahler Einstein metric on a compact complex surface is necessarily conformal to a K\"ahler metric and has positive scalar curvature. In particular, this result implies the statement of Theorem~\ref{thm einstein} for complex surfaces. However, the method of our proof works in all dimensions.

If $(M^{2n},g,J,\theta)$ is a compact lcK manifold and $g$ is Einstein with positive scalar curvature, then by Myers' Theorem and Remark \ref{cov}, $(M,g,J)$ is gcK. The classification of conformally K\"ahler compact Einstein manifolds in complex dimension $n\geq 3$ has been obtained by A.~ Derdzinski and G.~Maschler in a series of three papers \cite{DM2003,DM2005,DM2006}. They showed that the only examples are given by the construction of L. B\' erard-Bergery, \cite{Bergery1982}. 
In complex dimension $n=2$, the only compact complex surfaces which might admit proper gcK Einstein metrics are the blow-up of $\CM P^2$ at one, two or three points in general position, according to a result of C.~LeBrun, \cite[Theorem A]{LeBrun1997}.  Moreover, in the one point case, he showed that, up to rescaling and isometry, the only such metric is the well-known Page metric, \cite{Page1978}. The existence of a Hermitian Einstein metric on the blow-up of $\CM P^2$ at two different points was proven by X.~Chen, C.~LeBrun and B.~Weber in \cite{CLW2008}.

Theorem \ref{thm einstein} and the above remarks complete the proof of Theorem \ref{p1}.

\section{The holonomy problem for compact lcK manifolds}\label{shol}

Our next aim is to study compact lcK manifolds $(M,g,J,\theta)$ of complex dimension $n\ge 2$ with non-generic holonomy group: $\Hol_0(M,g)\subsetneq\SO(2n)$. By the Berger-Simons holonomy theorem, the following exclusive possibilities may occur:
\begin{itemize}
\item The restricted holonomy group $\Hol_0(M,g)$ is reducible;
\item $\Hol_0(M,g)$ is irreducible and $(M,g)$ is locally symmetric;
\item $M$ is not locally symmetric, and $\Hol_0(M,g)$ belongs to the following list: $\U(n)$, $\SU(n)$, $\Sp(n/2)$, $\Sp(n/2)\Sp(1)$, $\Spin(7)$ (for $n=4$).
\end{itemize}

\subsection{The reducible case}\label{sired}

In this section we classify the compact  lcK manifolds with reducible restricted holonomy. We start with the following result (for a proof see for 
instance the first part of the proof of \cite[Theorem 4.1]{BM2015}):
\begin{Lemma}\label{hol0}
If $(M,g)$ is a compact Riemannian manifold with 
$\Hol_0(M,g)$ reducible, then there exists a finite covering $\overline{M}$ of $M$, such that 
$\Hol(\overline{M},\bar g)$ is reducible, where $\bar g$ denotes the pull-back of $g$ to $\overline M$.
\end{Lemma}

Let now $(M,g,J,\theta)$ be a compact proper lcK manifold of complex dimension $n\ge2$ with $\Hol_0(M,g)$ reducible. Lemma \ref{hol0} shows that by replacing $M$ with some (compact) finite covering $\overline M$, and by pulling back the lcK structure to $\overline M$, one may assume that the tangent bundle can be decomposed as $\T M=D_1\oplus D_2$, where $D_1$ and $D_2$ are two parallel orthogonal oriented distributions of rank $n_1$, respectively $n_2$, with $2n=n_1+n_2$. By taking a further double covering if necessary, we may assume that the distributions are oriented. We first show that the case $n_1\ge2$ and $n_2\ge 2$ is impossible if the lcK structure is proper.

\begin{Theorem}\label{thmred}
Let $(M,g,J,\theta)$ be a compact  lcK manifold of complex dimension $n\ge2$. If there exist two orthogonal parallel oriented distributions $D_1$ and $D_2$, of respective ranks $n_1\ge 2$ and  $n_2\ge2$, such that $\T M=D_1\oplus D_2$, then $\theta\equiv0$.
\end{Theorem}

\begin{proof}
Since the arguments for $n=2$ and $n\ge3$ are of different nature, we treat the two cases separately. Consider first the case of complex dimension $n=2$. Then both distributions $D_1$ and $D_2$ have rank $2$, and their volume forms $\Omega_1$ and $\Omega_2$ define two K\"ahler structures on $M$ compatible with $g$ by the formula $g(I_\pm\cdot,\cdot)=\Omega_1\pm\Omega_2.$ Using the case $n=2$ in Theorem \ref{thm kaehler} below, we deduce that $J$ commutes with $I_+$ and with $I_-$. In particular, $J$ preserves the $\pm 1$ eigenspaces of $I_+I_-$, which are exactly the distributions $D_1$ and $D_2$. Since $J$ is also orthogonal, its restriction to $D_1$ and $D_2$ coincides up to sign with the restriction of $I_+$ to $D_1$ and $D_2$. Thus $J=\pm I_+$ or $J=\pm I_-$. In each case, the structure $(g,J)$ is K\"ahler, so $\theta\equiv 0$.

We consider now the case $n\geq 3$. Let $\theta=\theta_1+\theta_2$ be the corresponding splitting of the Lee form.
We fix a local orthonormal basis $\{e_i\}_{i=1,\ldots, 2n}$, which is parallel at the point where the computation is done and denote by $e_i^a$ the projection of $e_i$ onto $D_a$, for $a\in\{1,2\}$. 

The exterior differential and $\Omega$ split with respect to the decomposition of the tangent bundle as follows: $\dd=\dd_1+\dd_2$ and $\Omega=\Omega_{11} +2 \Omega_{12}+\Omega_{22}$, where for $a,b\in\{1,2\}$ we define:
\[\dd_a:=\sum_{i=1}^{2n}e_i^a\wedge \nabla_{e_i^a},\quad  
\Omega_{ab}:= \frac{1}{2}\sum_{i=1}^{2n} e_i^a\wedge (Je_i)^b= 
\frac{1}{2}\sum_{i=1}^{2n} e_i^a\wedge (Je_i^a)^b.\] 
The last equality follows for instance by considering a local orthonormal basis of $\T M$, whose 
first $n_1$ vectors are tangent to $D_1$.

\begin{Lemma}\label{lemm12}
With the above notation, for any vector fields $X_1\in D_1$ and $X_2\in D_2$, the following relations hold:
\begin{equation}\label{nabla12}
\nabla_{X_1} \theta_2=-\theta_1(X_1)\theta_2, \quad \nabla_{X_2} \theta_1=-\theta_2(X_2)\theta_1.
\end{equation}
\end{Lemma}

\begin{proof}
Note that $\dd \theta=0$ implies $\dd_a\theta_b+\dd_b\theta_a=0$, for all $a,b\in\{1,2\}$. For $c\in\{1,2\}$ we compute:
\begin{eqnarray*}
\dd_c\Omega_{ab}&\overset{\eqref{nablaJ}}{=}&\frac{1}{2}\sum_{i,j=1}^{2n} e_j^c\wedge e_i^a\wedge\left(\<e_j^c,e_i>J\theta-\<J\theta,e_i>e_j^c+\<J e_j^c,e_i>\theta-\<\theta,e_i>Je_j^c\right) ^b\\
&=&\frac{1}{2}\sum_{i=1}^{2n} \left(e_i^c\wedge e_i^a\wedge J\theta_b-(Je_i)^c\wedge e_i^a\wedge \theta_b-e_i^c\wedge\theta_a\wedge(Je_i^c)^b\right)
\\&=&\Omega_{ac}\wedge\theta_b+\Omega_{cb}\wedge\theta_a,
\end{eqnarray*}
so for all $a,b,c\in\{1,2\}$ we have
\begin{equation}\label{dc}
\dd_c\Omega_{ab}=\Omega_{ac}\wedge\theta_b+\Omega_{cb}\wedge\theta_a.
\end{equation}
Using the fact that $\dd_c^2=0$ for $c\in\{1,2\}$, we obtain
\begin{equation}\label{d2c}
0=\dd_c^2\Omega_{ab}=\Omega_{ac}\wedge(\dd_c\theta_b+\theta_c\wedge\theta_b)+\Omega_{cb}\wedge(\dd_c\theta_a+\theta_c\wedge\theta_a).
\end{equation}
For $c=a\neq b$ in \eqref{d2c}, we get $\Omega_{cc}\wedge(\dd_c\theta_b+\theta_c\wedge\theta_b)=0$ and for $a=b$: $\Omega_{cb}\wedge(\dd_c\theta_b+\theta_c\wedge\theta_b)=0$. Summing up, we obtain that $\Omega\wedge (\dd_c\theta_b+\theta_c\wedge\theta_b)=0$, which by the injectivity of $\Omega\wedge\cdot$ on manifolds of complex dimension greater than $2$, implies that $\dd_c\theta_b=-\theta_c\wedge\theta_b$. Applying this identity for $b\neq c$ to $X_c\in D_c$ and $X_b\in D_b$ yields \eqref{nabla12}.
\end{proof}

The symmetries of the Riemannian curvature tensor imply that $R_{X_1,X_2}=0$, and thus $R_{X_1,X_2} J=[R_{X_1,X_2}, J]=0$, for every $X_1\in D_1$ and $X_2\in D_2$. 

Using \eqref{RJ} for $X:=X_1$ and $Y:=X_2$ and applying Lemma~\ref{lemm12}, we obtain:
\begin{multline}\label{curb3}
0=\<X_1,\theta_1>X_2\wedge J\theta_1 -\<X_2,\theta_2>X_1\wedge J\theta_2
-\<X_2,\theta_2>JX_1\wedge \theta_2 +\<X_1,\theta_1>JX_2\wedge \theta_1 \\
-|\theta|^2X_2\wedge JX_1+|\theta|^2X_1\wedge JX_2
+X_2\wedge J\nabla_{X_1} \theta_1 +JX_2\wedge \nabla_{X_1} \theta_1-X_1\wedge J\nabla_{X_2}\theta_2 
 -JX_1\wedge \nabla_{X_2}\theta_2,
\end{multline}
for every $X_1\in D_1$ and $X_2\in D_2$.

\begin{Lemma}
The following formula holds:
\begin{equation}\label{n11}
\nabla_{X_1}\theta_1=-\<X_1,\theta_1>\theta_1+\frac1{n_1}\left(|\theta_1|^2-{\delta\theta_1}\right)X_1, \quad \forall X_1\in D_1.
\end{equation}
\end{Lemma}
\begin{proof}

Let $\mathcal{U}$ denote the open set $\mathcal{U}:=\{x\in M\, |\,\ (JD_2)_x\not\subset (D_1)_x\}$. By continuity, it is enough to prove the result on the open sets $M\setminus \overline{\mathcal{U}}$ and $\mathcal{U}$.

Let $\mathcal{O}$ be some open subset of $M\setminus \overline{\mathcal{U}}$, \emph{i.e.} at every point $x$ of $\mathcal{O}$ the inclusion $(JD_2)_x\subset (D_1)_x$ holds. On $\mathcal{O}$, let $X$ be some vector field  and $Y_2,Z_2$ vector fields tangent to $D_2$. By assumption, we have $JY_2\in D_1$, hence $\nabla_X JY_2\in D_1$ and $\nabla_XY_2\in D_2$, thus $J\nabla_X Y_2\in D_1$.
Applying \eqref{nablaJ}, we obtain 
\begin{eqnarray*}
0&=&\<\nabla_X JY_2, Z_2>-\<J\nabla_X Y_2,Z_2>=\<(\nabla_XJ)Y_2,Z_2>\\
&=&\<X,Y_2>(J\theta)(Z_2)-\<X,Z_2>(J\theta)(Y_2)-\<X,JY_2>\theta(Z_2)+\<X,JZ_2>\theta(Y_2).
\end{eqnarray*}
Since $n_2\geq 2$, for any $Y_2\in D_2$ there exists a non-zero $Z_2\in D_2$ orthogonal to $Y_2$. Taking $X=JZ_2\in D_1$ in the above formula yields  $\theta(Y_2)=0$. This shows that $\theta_2=0$, so $\theta=\theta_1$. Taking $X=Z_2\in D_2$ in the above formula yields $\theta_1(J Y_2)=0$, for all $Y_2\in D_2$. 
Substituting into \eqref{curb3}, we obtain for all $X_1\in D_1$ and $Y_2\in D_2$:
\begin{multline}\label{eqtr}
\<X_1,\theta_1>Y_2\wedge J\theta_1 +\<X_1,\theta_1>JY_2\wedge \theta_1
-|\theta_1|^2Y_2\wedge JX_1+|\theta_1|^2X_1\wedge JY_2
\\
+Y_2\wedge J\nabla_{X_1} \theta_1 
+JY_2\wedge \nabla_{X_1} \theta_1=0.
\end{multline}

Let us now consider the decomposition $D_1=JD_2\oplus D'_1$, where $D'_1$ denotes the orthogonal complement of $JD_2$ in $D_1$. Note that $D'_1$ is $J$-invariant, since it is also the orthogonal complement in $\T M$ of the $J$-invariant distribution $D_2\oplus JD_2$. Let $X_1=JV_2+V_1$ and $\nabla_{X_1}\theta_1=J W_2+W_1$ be the decomposition of $X_1$, respectively of $\nabla_{X_1}\theta_1$, with respect to this splitting, \emph{i.e.} $V_2, W_2\in D_2$ and $V_1,W_1\in D'_1$. As shown above, $\theta_1$ vanishes on $JD_2$, meaning that $\theta_1\in D'_1$.

Taking the trace with respect to $Y_2$ in \eqref{eqtr} yields 
\begin{equation}
n_2\<X_1,\theta_1>J\theta_1 
+|\theta_1|^2 [(n_2-1)V_2-n_2 JV_1]
+n_2 JW_1-(n_2-1)W_2=0,
\end{equation}
which further implies, by projecting onto $D_2$ and $D'_1$, 
that $W_1=-\<X_1,\theta_1>\theta_1  +|\theta_1|^2 V_1$ and $W_2=|\theta_1|^2 V_2$. 
Hence, $\nabla_{X_1} \theta_1= |\theta_1|^2 X_1-\<X_1,\theta_1>\theta_1$, 
which in particular implies $\delta\theta_1=(1-n_1)|\theta_1|^2$, 
proving \eqref{n11} on $M\setminus \overline{\mathcal{U}}$.

We further show that the formula \eqref{n11} holds on $\mathcal{U}$. At every point  $x$ of $\mathcal{U}$ there exist vectors $X_2, Y_2\in (D_2)_x$ such that $X_2\perp Y_2$ and $\<Y_2,JX_2>\neq 0$. Indeed, by definition there exists $Y_2\in (D_2)_x$ such that $JY_2\notin D_1$, and we can take $X_2$ to be the $D_2$-projection of $JY_2$.

For any vector $X_1\in (D_1)_x$ we take the scalar product with $X_1\wedge Y_2$ in \eqref{curb3} and obtain:
\begin{multline}\label{curb2}
\<JX_2,Y_2>\left(\<\nabla_{X_1} \theta_1,X_1>+|\<X_1,\theta_1>|^2\right)=\\
-|X_1|^2\left(\<X_2,\theta_2>\< J\theta_2,Y_2>
 -|\theta|^2 \<JX_2,Y_2>
 +\< J\nabla_{X_2}\theta_2, Y_2>\right).
\end{multline}
We thus get $\<\nabla_{X_1} \theta_1,X_1>+|\<X_1,\theta_1(x)>|^2=f_1(x)|X_1|^2,$ for every $X_1\in (D_1)_x$, 
where the real number $f_1(x)$ does not depend on $X_1$. By polarization, we obtain:
\begin{equation}\label{nabla11}
\nabla_{X_1}\theta_1=-\<X_1,\theta_1(x)>\theta_1(x)+f_1(x)X_1, \quad \forall X_1\in (D_1)_x.
\end{equation}
Taking the trace with respect to $X_1$ in this formula and using \eqref{nabla12} we obtain $(\delta\theta_1)_x=|\theta_1(x)|^2-n_1f_1(x)$, whence:
\begin{equation}\label{f1}
f_1(x)=\frac1{n_1}\left(|\theta_1|^2-{\delta\theta_1}\right)(x),\qquad\forall x\in \mathcal{U}.
\end{equation}
From \eqref{nabla11} and \eqref{f1} we obtain \eqref{n11} on $\mathcal{U}$. This proves the lemma.
\end{proof}

A similar argument yields 
\begin{equation}\label{n22}
\nabla_{X_2}\theta_2=-\<X_2,\theta_2>\theta_2+\frac1{n_2}\left(|\theta_2|^2-{\delta\theta_2}\right)X_2, \quad \forall X_2\in D_2.
\end{equation}

Substituting \eqref{n11} and \eqref{n22} into \eqref{curb3}, we obtain
\[
\left(\frac1{n_1}\left(|\theta_1|^2-{\delta\theta_1}\right)+\frac1{n_2}\left(|\theta_2|^2-{\delta\theta_2}\right)-|\theta|^2\right)\left(X_2\wedge JX_1-X_1\wedge JX_2\right)=0,\;\forall X_1\in D_1,\  X_2\in D_2.
\]
Note that for every $X_1\in D_1,\  X_2\in D_2$ the two-forms $X_2\wedge JX_1$ and $X_1\wedge JX_2$ are mutually orthogonal. So, choosing $X_1$ non-collinear to $JX_2$ (which is possible as $n_1\ge 2$), the 2-form appearing in the previous formula is non-zero. Hence, we necessarily have
\[
\frac1{n_1}\left(|\theta_1|^2-{\delta\theta_1}\right)+\frac1{n_2}\left(|\theta_2|^2-{\delta\theta_2}\right)-|\theta|^2=0.
\]
Integrating this relation over $M$, we get
\[\int_M |\theta|^2 \dd\mu_g=\frac{1}{n_1}\int_M |\theta_1|^2 \dd\mu_g+\frac{1}{n_2}\int_M |\theta_2|^2 \dd\mu_g.\]
Since $|\theta|^2=|\theta_1|^2+|\theta_2|^2$, we obtain
$\left(1-\frac{1}{n_1}\right)\int_M |\theta_1|^2 \dd\mu_g+\left(1-\frac{1}{n_2}\right)\int_M |\theta_2|^2 \dd\mu_g=0$.
As $n_1,n_2\geq 2$, it follows that $\theta\equiv0$. This concludes the proof of the theorem.
\end{proof}

\begin{Remark} For every $n\ge 2$, the tangent bundle $\T(\CM^n\setminus\{0\})$ endowed with the flat metric $g$ defined in Example \ref{expl} can be written as an orthogonal direct sum of two parallel distributions of ranks at least 2 in infinitely many ways, but the gcK structure $(g,J_0)$ on $\CM^n\setminus\{0\}$ has non-vanishing Lee form $\theta=-2\ln r$.
The compactness assumption in Theorem~\ref{thmred} is thus necessary.
\end{Remark}

It remains to consider the case when one of the two oriented parallel distributions has rank $1$, and is thus spanned by a globally defined parallel unit vector field. This case was studied by the second named author in \cite[Theorem~3.5]{Moroianu2015} for $n\ge3$. We  will give here a simpler proof of his result, which also extends it to the missing case $n=2$.

\begin{Theorem}[{\em cf.} {\cite[Theorem~3.5]{Moroianu2015}}] \label{main am} 
Let $(M,g,J,\theta)$ be a compact proper lcK manifold of complex dimension $n\geq 2$ admitting a non-trivial parallel vector field $V$. Then, the following exclusive possibilities occur:
\begin{enumerate}
\item[(i)] The Lee form $\theta$ is a non-zero constant multiple of $V^\flat$, so $M$ is a Vaisman manifold.
\item[(ii)] The Lee form $\theta$ is exact, so $(M,g,\Omega,\theta)$ is gcK, and there exists a complete simply connected K\"ahler manifold $(N,g_N,\Omega_N)$ of real dimension $2n-2$, a smooth non-constant real function $c:\RM\to\RM$ and a discrete co-compact group $\Gamma$ acting freely and totally discontinuously on $\RM^2\times N$, preserving the metric $\dd s^2+\dd t^2+e^{2c(t)}g_N$, the Hermitian $2$-form $\dd s\wedge\dd t+e^{2c(t)}\Omega_N$ and the vector fields $\partial_s$ and $\partial_t$, such that $M$ is diffeomorphic to $\Gamma\backslash(\RM^2\times N)$, and the structure $(g,\Omega,\theta)$ corresponds to $(\dd s^2+\dd t^2+e^{2c(t)}g_N, \dd s\wedge\dd t+e^{2c(t)}\Omega_N,\dd c)$ through this diffeomorphism.
\end{enumerate}
\end{Theorem}

\begin{proof}
Let $V$ be a parallel vector field of unit length on $M$. We identify as usual 1-forms with vectors using the metric $g$ and decompose the Lee form as $\theta=aV+bJV+\theta_0$, where $a:=\<\theta, V>$, $b:=\<\theta, JV>$  and $\theta_0$ is orthogonal onto $V$ and $JV$. We compute:
\begin{equation}\label{deltatheta}
\delta\theta=-V(a)-JV(b)+b\delta JV+\delta\theta_0.
\end{equation}
On the other hand, we have:
\begin{eqnarray*}\delta JV&=&-\sum_{i=1}^{2n}\<(\nabla_{e_i} J)V, e_i>\overset{\eqref{nablaJ}}{=}\sum_{i=1}^{2n} -\<(e_i\wedge J\theta+J e_i\wedge \theta)(V), e_i>\\
&=&(2-2n)\<\theta, JV>=(2-2n)b,\end{eqnarray*}
which together with \eqref{deltatheta} yields
\begin{equation}\label{delth}
\delta\theta=-V(a)-JV(b)+(2-2n)b^2+\delta\theta_0.
\end{equation}

Replacing $X$ by $V$ in \eqref{RJcontr} and using that $V$ is parallel, we obtain:
\begin{equation*}
(2n-3)\left(a J\theta-|\theta|^2JV+ J\nabla_{V} \theta\right)-b\theta -\nabla_{JV}\theta-JV\delta\theta=0.
\end{equation*}
Taking the scalar product with $JV$ yields
\begin{equation}\label{eq delth}
(2n-3)\left(a^2-|\theta|^2+ \<\nabla_{V} \theta,V> \right)-b^2   -\<\nabla_{JV}\theta,JV>-\delta\theta=0.
\end{equation}
Further, we compute
$$\<\nabla_V\theta, V>=V(\<\theta, V>)=V(a),$$
$$\<\nabla_{JV}\theta, JV>=JV(b)-\<\theta, (\nabla_{JV} J)V>\overset{\eqref{nablaJ}}{=} JV(b)-\<\theta, bJV+aV-\theta>=JV(b)+|\theta_0|^2,$$
which together with \eqref{delth} and \eqref{eq delth} imply that 
\begin{equation*}
(2n-2)(V(a)-|\theta_0|^2)=\delta\theta_0.
\end{equation*}
Integrating over $M$, we obtain $\int_M |\theta_0|^2\dd\mu_g=0$, because  
$\int_M V(a) \dd\mu_g= \int_M a\delta V \dd\mu_g=0$, as $V$ is parallel. Hence, $\theta_0=0$, showing that $\theta=aV+bJV$.

{\bf Claim.} The function $a$ is constant and $ab=0$.

{\it Proof of the Claim.}  Equation \eqref{nablaJ} yields
\begin{equation}\label{nablaJV}
\nabla_X JV= \<X, V>(-bV+aJV)+bX-\<X, JV>(aV+bJV)-aJX,
\end{equation}
which allows us to compute the exterior differential of $JV$, as follows:
\begin{equation}\label{ddJV}
\dd JV= 2a(V\wedge JV-\Omega).
\end{equation}
From the fact that $\theta$ is closed and $V$ is parallel, we obtain
$$0=\dd\theta=\dd a\wedge V+\dd b\wedge JV+b\, \dd JV=\dd a\wedge V+\dd b\wedge JV+2ab(V\wedge JV-\Omega),$$
which implies that $ab=0$, for instance, by taking the scalar product with $X\wedge JX$ for some vector field $X$ orthogonal to $V$ and $JV$. In particular, we have
\begin{equation}\label{eq da}
\dd a\wedge V+\dd b\wedge JV=0.
\end{equation}
Differentiating again \eqref{ddJV} yields
\begin{equation*}
0=\dd a\wedge (V\wedge JV-\Omega)+a(-V\wedge \dd JV -\dd \Omega)=\dd a\wedge (V\wedge JV-\Omega)-2abJV\wedge\Omega=\dd a\wedge (V\wedge JV-\Omega),
\end{equation*}
which together with \eqref{eq da} shows that $\dd a=0$, thus proving the claim.

If $a$ is non-zero, the second part of the claim shows that $b\equiv0$, so $\theta=aV$ is parallel and $(M,g,J, \theta)$ is Vaisman.

If $a=0$, Equation \eqref{nablaJV} becomes:
\[\nabla_X JV= b\left(X-\<X, V>V-\<X, JV>JV\right).\]
We conclude that in this case the metric structure on $M$ is given as in (ii) by applying Lemma~3.3 and Lemma~3.4 in \cite{Moroianu2015}.
\end{proof}

\begin{Corollary}\label{cor-hol}
Let $(M,g,J,\theta)$ be a compact proper lcK manifold of complex dimension $n\ge2$. If $(M,g)$ has reducible holonomy, then its restricted holonomy group $\Hol_0(M,g)$ is conjugate to $\SO(2n-1)$.
\end{Corollary}
\begin{proof}
By Lemma \ref{hol0}, Theorem \ref{thmred} and Theorem \ref{main am}, we need to distinguish two cases: 

{\bf Case 1.} If $(M,g,J,\theta)$ is Vaisman. Then the Lee form $\theta$ is parallel (and non-vanishing), so $(M,g)$ is locally isometric to $\RM\times S$ for some Riemannian manifold $(S,g_S)$. It is well known that $S$ is a Sasakian manifold, but since we want to avoid introducing this class of manifolds, we will derive the necessary formulas directly. 

As $\theta$ is parallel, we can rescale the metric of $M$ such that $|\theta|=1$. Equation \eqref{RJ} applied to vector fields $X,Y$ tangent to $S$ (i.e. orthogonal to $\theta$) then yields:
$$R_{X,Y}J=X\wedge JY-Y\wedge JX,\qquad\forall X,Y\in\ker(\theta).$$
In particular, applying this formula to $\theta$ (seen as vector field) and using the fact that $R_{X,Y}\theta=0$ gives
$$R_{X,Y}(J\theta)=(R_{X,Y}J)(\theta)=\<Y,J\theta>X-\<X,J\theta>Y,\qquad\forall X,Y\in\ker(\theta).$$
The metric dual $\xi$ of $J\theta$ is parallel in the direction of $\theta$, so it is actually a vector field on $S$, and the previous relation reads
\begin{equation}\label{rs}R^S_{X,Y}\xi=g_S(Y,\xi)X-g_S(X,\xi)Y,\qquad\forall X,Y\in\T S,
\end{equation}
where $R^S$ is the Riemannian curvature tensor of $(S,g_S)$. 

Assume, for a contradiction, that $\Hol_0(M,g)$ is strictly contained in $\SO(2n-1)$. Then the same holds for $\Hol_0(S,g_S)$, so by the Berger-Simons theorem, we have three possibilities: 
\begin{itemize}
\item $(S,g_S)$ has reducible holonomy; this would contradict Theorem \ref{thmred} since then $(M,g)$ would have a holonomy reduction with both factors of dimension at least 2.
\item $\Hol_0(S,g_S)$ belongs to the Berger list; the unique group in this list corresponding to an odd-dimensional manifold is $\G_2$ (for $2n-1=7$). However, a manifold with holonomy $\G_2$ is Ricci-flat, whereas $\Ric^S(\xi)=(2n-2)\xi$ by taking a trace in \eqref{rs}. This case is thus impossible too.
\item $(S,g_S)$ is locally symmetric. Then $R^S$ is parallel, so by taking a further covariant derivative in \eqref{rs} we get 
\begin{equation}\label{rsz}R^S_{X,Y}(\nabla_Z\xi)=g_S(Y,\nabla_Z\xi)X-g_S(X,\nabla_Z\xi)Y,\qquad\forall X,Y,Z\in\T S.
\end{equation}
On the other hand, from \eqref{nablaJ} we see that $\nabla_Z\xi=-JZ$ when $Z$ is orthogonal to $\theta$ and $J\theta$, so the set $\{\nabla_Z\xi\ |\ Z\in\T S\}$ is equal to the orthogonal of $\xi$ in $\T S$. From \eqref{rs} and \eqref{rsz} we thus obtain that $S$ has constant sectional curvature 1, i.e. it is locally isometric to the round sphere, and has maximal holonomy group $\Hol_0(S,g_S)=\SO(2n-1)$, which contradicts our assumption.
\end{itemize}

{\bf Case 2.} The universal covering of $(M,g)$ is isometric to a Riemannian product $\RM\times S$, where $S = \RM\times N$ has a warped product metric $g_S=\dd t^2 + e^{2c(t)}g_N$ with periodic, but non-constant, warping function $c$. Denoting for convenience $f(t):=e^{c(t)}$, one of the O'Neill formulas for the curvature of warped products (cf. \cite[p. 210]{Oneill1983}) reads:
\begin{equation}\label{rs3}R^S_{X,\partial_t}\partial_t=-\frac{\ddot f}f X,\qquad\forall X\in C^\infty(\T N).
\end{equation}

Assume now that $\Hol_0(M,g)=\Hol(S,g_S)$ is strictly contained in $\SO(2n-1)$. Like before, Theorem \ref{thmred} shows that $(S,g_S)$ has irreducible holonomy. 

Next, if $\Hol(S,g_S)$ belongs to the Berger list, then $S$ is a $\G_2$-manifold since it has odd dimension, and therefore is Ricci-flat. On the other hand, taking the trace in  \eqref{rs3} immediately shows that 
\begin{equation}\label{rs4}\Ric^S(\partial_t,\partial_t)=(1-2n)\frac{\ddot f}f.
\end{equation} 
Thus $\Ric^S=0$ implies $\ddot f=0$, which is impossible since $f$ is a non-constant periodic function.

It remains to treat the case where $(S,g_S)$ is an irreducible symmetric space. In particular $S$ is Einstein with Einstein constant $\lambda$ and from \eqref{rs4} we get $\ddot f=\frac{\lambda}{1-2n}f$. As $f$ is non-constant and periodic, we necessarily have $\lambda> 0$ and 
\begin{equation}\label{rs5}f(t)=\sin(\mu t+\nu)
\end{equation} 
for some real constants $\mu$ and $\nu$ with $\mu^2=\frac{\lambda}{2n-1}.$ This is a contradiction, since the periodic function $f=e^c$ does not vanish at any point of $\RM$. This shows that $\Hol_0(M,g)$ is conjugate to $\SO(2n-1)$, and thus finishes the proof.
\end{proof}

\subsection{The irreducible locally symmetric case}

In this section we show the following result:

\begin{Proposition}\label{prop sym}
Every compact irreducible locally symmetric lcK manifold $(M^{2n},g,J,\theta)$ has vanishing Lee form. 
\end{Proposition}

\begin{proof}
An irreducible locally symmetric space is Einstein. If the scalar curvature of $M$ is non-positive, the result follows directly from Theorem \ref{thm einstein}.

Assume now that $M$ has positive scalar curvature. By Myers' Theorem and Remark~\ref{cov}, $(M,g,J)$ is gcK, so $\theta=\dd\varphi$ for some function $\f$, and $g_K:=e^{-2\varphi}g$ is a K\"ahler metric. 
Let $X$ be a Killing vector field of $g$. Then $X$ is a conformal Killing vector field of the metric $g_K$. By a result of Lichnerowicz \cite{Lichnerowicz1957} and Tashiro \cite{Tashiro1962}, every conformal Killing vector field with respect to a K\"ahler metric on a compact manifold is Killing. This shows that $X$ is a Killing vector field for both conformal metrics $g$ and $g_K$, hence $X$ preserves the conformal factor, \emph{i.e.} $X(\varphi)=0$. As $(M,g)$ is homogeneous and $X(\varphi)=0$ for each Killing vector field $X$ of $g$, it follows that the function $\varphi$ is constant. Thus $\theta=\dd\f=0$.
\end{proof}

\subsection{Compact irreducible lcK manifolds with special holonomy}\label{sec hol}

We finally consider compact lcK manifolds $(M,g,J,\theta)$ 
of complex dimension $n\ge2$, whose restricted holonomy group $\Hol_0(M,g)$ is in the Berger list. The following cases occur:

If $\Hol_0(M,g)$ is one of $\SU(n)$, $\Sp(n/2)$, or $\Spin(7)$ (for $n=4$), the metric $g$ is Ricci-flat and $\theta\equiv0$ by Theorem \ref{thm einstein}.

If $\Hol_0(M,g)=\Sp(n/2)\Sp(1)$, the metric $g$ is quaternion-K\"ahler, hence Einstein with either positive or negative scalar curvature. 
In the negative case one has  $\theta\equiv0$ by Theorem \ref{thm einstein}.
On the other hand, P.~Gauduchon, A.~Moroianu and U.~Semmelmann, have shown in \cite{GMS2011}, that the only compact quaternion-K\"ahler manifolds of positive scalar curvature which carry an almost complex structure are the complex Grassmanians of 2-planes, which are symmetric, thus again $\theta\equiv0$ by Proposition \ref{prop sym}. 

The case $\Hol_0(M,g)=\U(n)$ is more involved and will be treated in the next two sections.

\section{K\"ahler structures on lcK manifolds}\label{s3}

We now consider the case left open in the previous section, namely compact proper lcK manifolds $(M,g,J,\theta)$ whose restricted holonomy group is equal to $\U(n)$. We will see that there are examples of such structures, but they cannot be strictly lcK. In particular, the Riemannian metric $g$ of a compact strictly lcK manifold $(M,g,J,\theta)$ cannot be K\"ahler with respect to any complex structure on $M$.

The universal covering $(\widetilde M,\tilde g)$ has holonomy $\Hol(\widetilde M,\tilde g)=\U(n)$, so $\tilde g$ is K\"ahler with respect to some complex structure $\tilde I$. Every deck transformation $\gamma$ of $\widetilde M$ is an isometry of $\tilde g$, so $\gamma^*\tilde I$ is parallel with respect to the Levi-Civita connection of $\tilde g$. As $\Hol(\widetilde M,\tilde g)=\U(n)$, we necessarily have $\gamma^*\tilde I=\pm \tilde I$ for every $\gamma\in\pi_1(M)\subset \mathrm{Iso}(\widetilde M)$. The group of $\tilde I$-holomorphic deck transformations is thus a subgroup of index at most 2 of $\pi_1(M)$, showing that after replacing $M$ with some double covering if necessary, there exists an integrable complex structure $I$, such that $(M,g,I)$ is a K\"ahler manifold.  



\begin{Theorem}\label{thm kaehler}
Let $(M,g,J,\theta)$ be a compact proper lcK manifold of complex dimension $n\ge2$ carrying a complex structure $I$, such that $(M,g,I)$ is a K\"ahler manifold. Then $I$ commutes with $J$ and $(M,g,J,\theta)$ is globally conformally K\"ahler. 
\end{Theorem}
\begin{proof}
The Riemannian curvature tensor of $(M,g)$ satisfies $R_{X,Y}=R_{IX,IY}$, so in particular we have $R_{X,Y} J=R_{IX,IY} J$, for all vector fields $X$ and $Y$. Using \eqref{RJ}, this identity implies that
\begin{multline*}
\<X,\theta>Y\wedge J\theta -\<Y,\theta>X\wedge J\theta
-\<Y,\theta>JX\wedge \theta +\<X,\theta>JY\wedge \theta -|\theta|^2Y\wedge JX+|\theta|^2X\wedge JY\\
+Y\wedge J\nabla_{X} \theta+JY\wedge \nabla_{X} \theta -X\wedge J\nabla_{Y}\theta 
 -JX\wedge \nabla_{Y}\theta\\
=\<IX,\theta>IY\wedge J\theta -\<IY,\theta>IX\wedge J\theta
-\<IY,\theta>JIX\wedge \theta +\<IX,\theta>JIY\wedge \theta\\ -|\theta|^2IY\wedge JIX+|\theta|^2IX\wedge JIY
+IY\wedge J\nabla_{IX} \theta +JIY\wedge \nabla_{IX} \theta-IX\wedge J\nabla_{IY}\theta 
 -JIX\wedge \nabla_{IY}\theta,
\end{multline*}
for all vector fields $X,Y$. Let  $\{e_i\}_{i=1,\ldots,2n}$ be a local orthonormal basis of $\T M$, which is parallel at the point where the computation is done. Taking the interior product with $X$ in the above identity and summing over $X=e_i$, we obtain:
\begin{multline}\label{contr1}
(4-2n)\<Y,\theta> J\theta  
+\<JY,\theta>\theta 
+(2n-4)|\theta|^2 JY +(4-2n) J\nabla_{Y}\theta +\nabla_{JY} \theta+\delta\theta JY\\
=\<IY,\theta>IJ\theta-\<IJ\theta,\theta>IY-\<IY,\theta>\tr(JI)\theta+\<IY,\theta>JI\theta
+\<\theta, IJIY> \theta\\
+|\theta|^2\tr(IJ)IY -|\theta|^2IJIY
-\sum_{i=1}^{2n}\<e_i,J\nabla_{Ie_i}\theta>IY+IJ\nabla_{IY}\theta+\nabla_{IJIY} \theta
-\tr(JI)\nabla_{IY}\theta+JI \nabla_{IY}\theta.
\end{multline}
Substituting  $Y=e_j$ in \eqref{contr1}, taking the scalar product with $Je_j$ and summing over $j=1,\ldots,2n$ yields:
\begin{multline}\label{contr2}
(4n^2-10n+6)|\theta|^2+(4n-6)\delta\theta=
-2\tr(IJ)\<I\theta,J\theta>-
2\<IJ\theta,JI\theta>\\
-(\tr (IJ))^2|\theta|^2+\tr(IJIJ)|\theta|^2
+2\tr(IJ)\sum_{i=1}^{2n}\<IJ\nabla_{e_i}\theta,e_i>
-2\sum_{i=1}^{2n}\<IJIJ\nabla_{e_i}\theta,e_i>.
\end{multline}

By a straightforward computation, using \eqref{nablaJ} and the fact that $\nabla\theta$ is a symmetric endomorphism, we have the following identities:
$$
\tr(IJ)\sum_{i=1}^{2n}\<e_i,IJ\nabla_{e_i}\theta>=
-\delta(\tr(IJ)JI\theta)
+(2n-2)\tr(IJ)\<I\theta,J\theta>
-2\<IJ\theta,JI\theta>+2|\theta|^2,$$
$$
\sum_{i=1}^{2n}\<e_i,IJIJ\nabla_{e_i}\theta>=
-\delta(JIJI\theta)
-(2n-3)\<IJ\theta,JI\theta>
+\tr(IJ)\<I\theta,J\theta>
+|\theta|^2.
$$
Substituting these in \eqref{contr2}, we obtain
\begin{multline}\label{eqn theta}
(4n^2-10n+4)|\theta|^2-(4n-8)\tr(IJ)\<I\theta,J\theta>
-(4n-12)\<IJ\theta,JI\theta>\\
-\tr(IJIJ)|\theta|^2+(\tr(IJ))^2|\theta|^2=-(4n-6)\delta \theta-2\delta(\tr(IJ)JI\theta)
+2\delta(JIJI\theta).
\end{multline}

In order to exploit this formula we need to distinguish two cases.

{\bf Case 1:} If $n=2$, \eqref{eqn theta} becomes:
\begin{equation}\label{eqn theta2}
4\<IJ\theta,JI\theta>-\tr(IJIJ)|\theta|^2+(\tr(IJ))^2|\theta|^2=-2\delta( \theta+\tr(IJ)JI\theta
-JIJI\theta).
\end{equation}

We claim that $I$ and $J$ define opposite orientations on $\T M$.
Assume for a contradiction that they define the same orientation, and recall that complex structures compatible with the orientation on an oriented $4$-dimensional Euclidean vector space may be identified with imaginary quaternions of norm $1$ acting on $\HM$  by left multiplication.  For any $q,v\in \mathbb{H}$, we have: $\<q v, v>=\frac{1}{4}\tr(q)|v|^2$, where $\tr(q)$ denotes the trace of $q$ acting by left multiplication on $\HM$. 
For every $x\in M$ we can identify $\T _xM$ with $\HM$ and view $I,J$ as unit quaternions acting by left multiplication. The previous relation gives the following pointwise equality: $4\<IJIJ\theta,\theta>=\tr(IJIJ)|\theta|^2$. Substituting in \eqref{eqn theta2} and integrating over $M$, implies that $\tr(IJ)=0$, so $I$ and $J$ anti-commute.  Equation \eqref{eqn theta2} then further implies that $\delta\theta=0$. Replacing these two last equalities in \eqref{contr2} yields $|\theta|^2=0$. This contradicts the assumption that the lcK structure $(g,J,\theta)$ is proper. Hence, $I$ and $J$ define opposite orientations and thus they commute.

Since $(M,g,I)$ is a compact K\"ahler manifold, it follows that its first Betti number is even. N.~Buchdahl~\cite{Buchdahl1999} and A.~Lamari~\cite{Lamari1999} proved that each compact complex surface with even first Betti number carries a K\"ahler metric. On the other hand, I.~Vaisman proved that if a complex manifold $(M,J)$ admits a $J$-compatible K\"ahler metric, then every lcK metric on $(M,J)$ is gcK \cite[Theorem 2.1]{Vaisman1980}. This shows that $\theta$ is exact.

{\bf Case 2.} We assume from now on that $n\geq 3$.
The integral over the compact manifold $M$ of the left hand side of \eqref{eqn theta} is zero, since the right hand side is the co-differential of a $1$-form. On the other hand, the following inequalities hold:
\begin{equation}\label{i1}
-(4n-8)\tr(IJ)\<I\theta,J\theta>\geq -(4n-8)|\tr(IJ)||\theta|^2\geq -\left((\tr (IJ))^2+(2n-4)^2\right)|\theta|^2,
\end{equation}
\begin{equation}\label{i2}
-(4n-12)\<IJ\theta,JI\theta>\geq-(4n-12) |\theta|^2,
\end{equation}
(it is here that the assumption $n\ge 3$ is needed), and
\begin{equation}\label{i3}
-\tr(IJIJ)|\theta|^2\geq -2n|\theta|^2.
\end{equation}

Summing up the inequalities \eqref{i1}--\eqref{i3} shows that the left hand side of \eqref{eqn theta} is non-negative. As the right hand side of  \eqref{eqn theta} is a divergence, we deduce that both terms vanish identically, and thus equality holds in \eqref{i1}--\eqref{i3}. 

Let $M'$ denote the set of points where $\theta$ is not zero and let $M''$ denote the interior of $M\setminus M'$. The open set $M'\cup M''$ is clearly dense in $M$. At each point of $M'$, the fact that equality holds in \eqref{i3} shows that $(IJ)^2=\mathrm{Id}$. Moreover, the endomorphism $(IJ)^2$ is $\nabla$-parallel along $M''$ (since $\theta=0$ along $M''$, so $(M'',g,J)$ is K\"ahler). We deduce that $(IJ)^2$ is $\nabla$-parallel along $M'\cup M''$, thus along the whole of $M$ by density. Moreover $M'$ is not empty (since by assumption the lcK structure $(g,J,\theta)$ is proper). As $(IJ)^2=\mathrm{Id}$ on $M'$, we finally get $(IJ)^2=\mathrm{Id}$ on $M$, which amounts saying that $I$ and $J$ commute at each point of $M$.

Moreover, the fact that equality holds in \eqref{i1} shows that for each point $x\in M'$ there exists $\e_x\in\{-1,1\}$ with $I\theta=\e_x J\theta$ and $\tr (IJ)=\e_x(2n-4)$. The function $\tr (IJ)$ is thus locally constant on $M'$ and on $M''$ (since as before $IJ$ is parallel along $M''$), so by density, it is constant on $M$. 
After replacing $I$ with $-I$ if necessary we may thus assume that $\tr(IJ)=2n-4$, and $I\theta=J\theta$ on $M$ (this last relation holds tautologically on $M\setminus M'$).

This shows that that the orthogonal involution $IJ$ has two eigenvalues: 1 with multiplicity  $2n-2$ and $-1$ with multiplicity $2$. 
At each point of $M'$, since $\theta$ and $I\theta$ are eigenvectors of $IJ$ for the eigenvalue $-1$, it follows that $IJX=X$, for every $X$ orthogonal on $\theta$ and $J\theta$, which can also be expressed by the formula
\begin{equation}\label{J}
JX=-IX+\frac2{|\theta^2|}\left(\<X,\theta>I\theta-\<X,I\theta>\theta\right), \qquad\forall X\in \T M'. 
\end{equation}
We thus have $\Omega^J=-\Omega^I+\frac2{|\theta^2|}\theta\wedge I\theta$ on $M'$. In particular, we have \begin{equation}\label{to}\theta\wedge\Omega^J=-\theta\wedge\Omega^I,
\end{equation}
at every point of $M$ (as this relation holds tautologically on $M\setminus M'$, where by definition $\theta=0$).
From \eqref{dOmega} and \eqref{to} we get
\begin{equation}\label{dOmegaJ}\dd\Omega^J=2\theta\wedge\Omega^J=-2\theta\wedge\Omega^I=-2{\mathrm L}^I(\theta),
\end{equation}
where ${\mathrm L}^I: \Lambda^* M\to \Lambda^* M$, ${\mathrm L}^I(\alpha):=\Omega^I\wedge\alpha$ is the Lefschetz operator of the K\"ahler manifold $(M,g,I)$. 

Using the Hodge decomposition on $M$, we decompose the closed $1$-form $\theta$ as $\theta=\theta_{H}+\dd\varphi$,
where $\theta_H$ is the harmonic part of $\theta$ and $\varphi$ is a smooth real-valued function on $M$. From \eqref{dOmegaJ} and the fact that ${\mathrm L}^I$ commutes with the exterior differential, we obtain 
\begin{equation}\label{L}{\mathrm L}^I(\theta_H)=-\dd\left(\frac12\Omega^J+{\mathrm L}^I\varphi\right).
\end{equation}
Moreover, since ${\mathrm L}^I$ commutes on any K\"ahler manifold with the Laplace operator (see \emph{e.g.} \cite{Moroianu2007B}), the left-hand side of \eqref{L} is a harmonic form and the right-hand side is exact. This implies that ${\mathrm L}^I\theta_H$ vanishes, so $\theta_H=0$ since ${\mathrm L}^I$ is injective on 1-forms for $n\geq 2$. Thus $\theta=\dd\varphi$ is exact, so $(M,g,J,\theta)$ is globally conformally K\"ahler.
\end{proof}

\begin{Example}
As in Example~\ref{expl}, we consider on $M:=\mathbb{C}^n\setminus\{0\}$ the standard flat structure $(g_0,J_0)$. Let $J$ be a constant complex structure on $M$, compatible with $g_0$ and which does not commute with $J_0$. Then, $(M, g:=r^{-4}g_0, J)$ is gcK and  
$(M,g, I)$ is K\"ahler, where $I$ is the pull-back of $J_0$ through the inversion, but $J$ and $I$ do not commute. This example shows that the compactness assumption in Theorem~\ref{thm kaehler} is necessary.
\end{Example}

\section{Conformal classes with non-homothetic K\"ahler metrics}\label{skahler}

As an application of Theorem \ref{thm kaehler}, we will describe in this section all compact conformal manifolds $(M^{2n},c)$ with $n\ge2$, such that the conformal class $c$ contains two non-homothetic K\"ahler metrics.

We start by constructing a class of examples, which will be referred to as the {\em Calabi Ansatz}.

\begin{Proposition}\label{prop}
Let $(N,h,J_N,\Omega_N)$ be a Hodge manifold, {\em i.e.} a compact K\"ahler manifold with $[\Omega_N]\in H^2(N,2\pi\ZM)$. 
Let $\pi:S\to N$ be the principal $S^1$-bundle with the connection (given by Chern-Weil theory) 
whose curvature form is the pull-back to $S$ of $i\Omega_N$. For any positive real number $\ell$, 
let $h_\ell$ be the unique Riemannian metric on $S$ such that $\pi$ is a Riemannian submersion 
with fibers of length $2\pi \ell$. For every $b>0$ and smooth function 
$\ell:(0,b)\to \RM^{>0}$, consider the metric $g_{\ell}:=h_{\ell(r)}+\dd r^2$ on 
$M':=S\times(0,b)$. Then the metric $g_{\ell}$ is globally conformally K\"ahler with respect to two non-conjugate
complex structures $J_+$, $J_-$ on $M'$. Moreover, if $\ell^2(r)=r^2(1+A(r^2))$ near $r= 0$ 
and $\ell^2(r)=(b-r)^2(1+B((b-r)^2))$ near $r=b$ for smooth functions $A,B$ 
defined near $0$ with $A(0)=B(0)=0$, then the metric completion $M$ of $(M',g_\ell)$ is a smooth manifold diffeomorphic to the total space of an $S^2$-bundle over $N$, and $g_{\ell}$, $J_+$, and $J_-$ extend smoothly to $M$.
\end{Proposition}
\begin{proof}
Let $i\omega\in\Omega^1(S,i\RM)$ denote the connection form on $S$ satisfying 
\begin{equation}\label{do}\dd\omega=\pi^*(\Omega_N).
\end{equation}
The metric $h_\ell$ is defined by $h_\ell:=\pi^*h+\ell^2\omega\otimes\omega.$
Let $\xi$ denote the vector field on $S$ induced by the $S^1$-action. By definition $\xi$ verifies $\pi_*\xi=0$ and $\omega(\xi)=1$. Let $X^*$ denote the horizontal lift of a vector field $X$ on $N$ (defined by $\omega(X^*)=0$ and $\pi_*(X^*)=X$). By the equivariance of the connection we have $[\xi,X^*]=0$ for every vector field $X$, and from \eqref{do} we readily obtain $[X^*,Y^*]=[X,Y]^*-\Omega_N(X,Y)\xi$. The Koszul formula immediately gives the covariant derivative $\nabla^\ell$ of the metric $g_\ell:=h_{\ell(r)}+\dd r^2$ on $M':=S\times(0,b)$:
\begin{eqnarray*}
\nabla^\ell_\xi\dr&=&\nabla^\ell_\dr\xi=\frac{\ell'}{\ell}\xi,\\
\nabla^\ell_\xi\xi&=&-\ell\ell'\dr,\\
\nabla^\ell_\dr\dr&=&\nabla^\ell_{X^*}\dr=\nabla^\ell_\dr X^*=0,\\
\nabla^\ell_{X^*}\xi&=&\nabla^\ell_\xi X^*=\frac{\ell^2}2(J_NX)^*,\\
\nabla^\ell_{X^*}Y^*&=&(\nabla^h_XY)^*-\frac12\Omega_N(X,Y)\xi.
\end{eqnarray*}
We now define for $\e=\pm1$ the Hermitian structures $J_\e$ on $(M',g_\ell)$ by 
\[J_\e(X^*):=\e (J_NX)^*,\qquad J_\e(\xi):=\ell\dr,\qquad J_\e(\dr):=- \ell^{-1}\xi.\]
A straightforward calculation using the previous formulas yields $\nabla^\ell_Z J_\e=Z\wedge J_\e\theta_\e+J_\e Z\wedge\theta_\e$ for every vector field $Z$ on $M$, where $\theta_\e:=\frac12\e \ell\dd r$. Thus $(g_\ell,J_\e)$ are globally conformally K\"ahler structures on $M'$ with Lee forms
\[\theta_\e=\e\dd\f,\qquad\hbox{where}\qquad \f(r):=\frac12\int_0^r\ell(t)\dd t.\]

The last statement of the proposition follows from a coordinate change (from polar to Euclidean coordinates) in the fibers $S^1\times (0,b)$ of the Riemannian submersion $M'\to N$. 
Indeed, in a neighbourhood of $r=0$, with Euclidean 
coodinates $x_1:=r\cos t$ and $x_2:=r\sin t$, we have:
\begin{equation*}
\begin{pmatrix}
 \partial_r\\ \frac{1}{r}\xi
\end{pmatrix}=
\begin{pmatrix}
\phantom{-}  \cos t  & \sin t \\ 
-\sin t & \cos t 
\end{pmatrix}
\begin{pmatrix}
 \partial_{x_1}\\ 
\partial_{x_2}
\end{pmatrix},
\end{equation*}
where $\xi=\partial_t$. In these coordinates, 
we have the following formulas for the complex structures and the metric:
\begin{gather*}
J_\e(\partial_{x_1})=\frac{\ell}{r}\left(
\frac{-A(r^2)}{\ell^2}x_1x_2\partial_{x_1}-
\left(\frac{A(r^2)}{\ell^2}x_1^2+1\right)\partial_{x_2}\right),\\
J_\e(\partial_{x_2})=\frac{\ell}{r}\left(
\frac{A(r^2)}{\ell^2}x_1x_2\partial_{x_2}+\left(1-
\frac{A(r^2)}{\ell^2}x_2^2\right)\partial_{x_1}\right),\\
g=\pi^*h+\left(1+\frac{A(r^2)}{r^2}x_2^2\right)\dd x_1^2+
\left(1+\frac{A(r^2)}{r^2}x_1^2\right)\dd x_2^2-2\frac{A(r^2)}{r^2}x_1x_2\dd x_1\dd x_2.
\end{gather*}
From the assumption on $A$, the functions $\frac{\ell}r$ and $\frac{A(r^2)}{\ell^2}$ extend smoothly at $r=0$, therefore, the complex structures $J_-$, $J_+$ and the metric $g$ extend smoothly at $r=0$. 
The same argument applies to the other extremal point $r=b$.
Hence, the metric $g_\ell$ on $M'$ extends to a smooth metric $g_0$ on $M$, and there exist two distinct K\"ahler structures on $M$ in the conformal class $[g_0]$, whose restrictions to $M'$ are equal to $(g_+:=e^{\f}g_\ell,J_+)$ and $(g_-:=e^{-\f}g_\ell,J_-)$ respectively.
\end{proof}

Conversely, the Calabi Ansatz can be characterized geometrically by the following data:

\begin{Proposition}\label{thmconv}
Let  $(M,g_0,I)$ be a compact globally conformally K\"ahler manifold with non-trivial Lee form $\theta_0=d\varphi_0$ and denote by $\nabla^0$ the Levi-Civita connection of $g_0$. We assume that  on $M'$, the set where $\theta_0$ is not vanishing, its derivative is given by:
\begin{equation} \label{der0theta}
\nabla^0_X  \theta_0=f\left(\theta_0(X)\theta_0+I\theta_0(X)I\theta_0\right), \quad \forall X\in \T M',
\end{equation}
for some function $f\in C^{\infty}(M')$. We denote by $\xi$ the metric dual of $I\theta_0$ with respect to $g_0$ and further assume that there exists a distribution $\mathcal{V}$ on $M$, such that $\mathcal{V}_x$ is spanned by $\xi$ and $I\xi$, for every $x\in M'$. Then
$(M,g_0)$ is obtained from the Calabi Ansatz.
\end{Proposition}

\begin{proof}
We first notice that $M'\neq M$. Indeed, $\theta_0$ vanishes at the extrema of the function $\varphi_0$ defined on the compact manifold $M$.

From \eqref{der0theta} and \eqref{nablaJ}  we deduce the following formulas on $M'$:
\begin{equation} \label{der0Jxi}
\nabla^0_X (I\xi)=-f(\<X,I\xi>I\xi+\<X,\xi>\xi), \quad \forall X\in \T M',
\end{equation}
\begin{equation}\label{der0xi}
\nabla^0_X\xi=(1+f)(\<X,\xi>I\xi-\<X,I\xi>\xi)-|\xi|^2IX,\quad \forall X\in \T M',
\end{equation}
which imply that the distribution $\mathcal{V}$ is totally geodesic along $M'$.

Equation \eqref{der0xi}  also shows that $\nabla^0\xi$ is a skew-symmetric endomorphism, hence $\xi$ is a Killing vector field on $(M',g_0)$. Since $\xi$ is tautologically Killing on the interior of $M\setminus M'$, it is Killing on the whole of $M$ by density.
We denote by $N$ one of the connected components of the zero set of $\xi$, which is thus a compact totally geodesic submanifold of $M$. Applying \eqref{der0xi} at a sequence of points of $M'$ converging to some point of $N$, we see that $\dd\xi^\flat$ has rank at most $2$ at each point of $N$. Moreover $\xi$ is not identically $0$, thus showing that $N$ has co-dimension $2$, and its normal bundle equals $\mathcal{V}|_N$.

Let $\Phi_s$ denote the 1-parameter group of isometries of $(M,g_0)$ induced by $\xi$ and let us fix some $p\in N$. For every $s\in\RM$, the differential of $\Phi_s$ at $p$ is an isometry of $\T_pM$ which fixes $\T_pN$, so it is determined by a rotation of angle $k(s)$ in $\cV_p$. From $\Phi_s\circ\Phi_{s'}=\Phi_{s+s'}$ we obtain $k(s)=ks$, for some $k\in\RM^*$. For $s_0=2\pi/k$, the isometry $\Phi_{s_0}$ fixes $p$ and its differential at $p$ is the identity. We obtain that $\Phi_{s_0}=\mathrm{Id}_M$, so $\xi$ has closed orbits. Note that any $p\in N$ is a fixed point of $\Phi_{s}$, for all $s\in\mathbb{R}$, and that $\Phi_{\frac{s_0}{2}}$ is an orientation preserving isometry whose differential at $p$ squares to the identity, and is the identity on $\T_pN=\mathcal{V}_p^\perp$. Hence,  ${(\dd \Phi_{\frac{s_0}{2}})_p}|_{\mathcal{V}_p}$ is either plus or minus the identity of ${\mathcal{V}_p}$. The first possibility would contradict the definition of $s_0$, so we have
\begin{equation} \label{rot}
(\dd \Phi_{\frac{s_0}{2}})_p|_{\mathcal{V}_p}=-\mathrm{Id}_{\mathcal{V}_p}.
\end{equation}

Let $\gamma$ be a geodesic of $(M,g_0)$ starting from $p$, such that $V:=\dot\gamma(0)\in\cV_p$ and $|\dot\gamma(0)|=1$. Since $\cV$ is totally geodesic, $\dot\gamma(t)\in\cV$ for all $t$. The function $g_0(\xi,\dot\gamma)$ clearly vanishes at $t=0$ and its derivative along $\gamma$ equals $g_0(\nabla^0_{\dot\gamma}\xi,\dot\gamma)=0$, so $g(\xi,\dot\gamma)\equiv0$ along $\gamma$. We thus have 
\begin{equation}\label{fc}I\xi_{\gamma(t)}=c_{p,V}(t)\dot\gamma(t),
\end{equation}
for some function $c_{p,V}:\RM\to\RM$. Clearly $c_{p,V}^2(t)=|\xi_{\gamma(t)}|^2$, so $c_{p,V}$ is smooth at all points $t$ with $\gamma(t)\in M'$.  By \eqref{der0Jxi}--\eqref{der0xi} we easily check that $[\xi,I\xi]=0$ on $M'$, and thus on $M$ by density. Hence, each isometry $\Phi_s$ preserves $I\xi$. Moreover, $\Phi_s(\gamma(t))$ is the geodesic starting at $p$ with tangent vector $(\Phi_s)_*(\dot\gamma(0))$. This shows that the function $c_p:=c_{p,V}$ does not depend on the unit vector $V$ in $\cV_p$ defining $\gamma$. 

We claim that in fact, for all $p,q\in N$,  $c_p(t)=c_q(t)$, for all $t$. In other words, the norm of $\xi_{\gamma(t)}$ only depends on $t$ and not on the initial data of $\gamma$ starting in $N$. For a fixed $t\in\mathbb{R}$, we consider the map $F \colon SN \to M$, $F (V):=\exp(tV)$, where $SN$ denotes the unit normal bundle of $N$. By the Gauss' Lemma, we know that $\dd F_V(T_V SN)\subset ({\dot\gamma_{p,V}}(t))^{\perp}$, where $\gamma_{p,V}$ denotes the geodesic starting at $p$ with unit speed vector $V$. Since $\xi$ is Killing, the function $g_0(\dot\gamma_{p,V}, \xi)$ is constant along $\gamma_{p,V}$ and thus identically zero, because $\xi$ vanishes on $N$. As $\dot\gamma_{p,V}\in \mathcal{V}$, it follows that  $\dot\gamma_{p,V}$ is proportional to $I\xi$, which is the metric dual of $-\theta_0$. On the other hand,  \eqref{der0theta} immediately gives $\dd|\theta_0|^2=2f|\theta_0|^2\theta_0$. Therefore, $\dd|\theta_0|^2$ vanishes on $\dd F_V(T_V SN)$, showing that the norm of $\xi_{\gamma(t)}$ does not depend on the starting point either. Hence, we further denote the function $c_p=c_{p,V}$ simply by $c\colon\mathbb{R}\to\mathbb{R}$.

Differentiating the relation $\gamma_{p,V}(t)=\gamma_{p,-V}(-t)$ which holds for all geodesics and for all $t$, yields $\dot\gamma_{p,V}(t)=-\dot \gamma_{p,-V}(-t)$. Therefore, from \eqref{fc} we conclude that $c(-t)=-c(t)$, for all $t$. Moreover, $c(t)$ is non-vanishing for $|t|\ne 0$ and sufficiently small. By replacing $I$ with $-I$ if necessary, we thus can assume that $c$ is negative on some interval $(0,\e)$ and positive on $(-\e,0)$. Since $(\varphi_0(\gamma(t)))'=\theta_0(\dot\gamma(t))=-c(t)$, we conclude that 
$N$ is a connected component of the level set of a local minimum of $\varphi_0$.

By compactness of $N$, the exponential map defined on the normal bundle of $N$ is surjective, so its image contains points where $\f_0$ attains its absolute maximum. At such a point, the vector field $\xi$ vanishes, so \eqref{fc} shows that $t_0:=\inf\{t>0\, \vert\, c(t)=0\}$ is well-defined and positive. Let $N'$ be a connected component of 
the inverse image through $\f_0$ of $\f_0(\exp_p(t_0V))$,  for some $p\in N$ and some unit vector $V$ in $ \mathcal{V}_{p}$. The above argument, applied to $N'$ instead of $N$, shows that $N'$ is a connected component of the level set of a local maximum of $\varphi_0$. It also shows that  $\exp_q(t_0W)\in N'$ for any $q\in N$ and any unit vector $W\in \mathcal{V}_q$. From \eqref{rot} it follows that $\dot \gamma_{p,-V}(t_0)=-\dot\gamma_{p,V}(t_0)$, for any $p\in N$ and any unit vector $V\in \mathcal{V}_p$. In other words, if a geodesic starting at a point $p$ of $N$ with unit speed vector $V\in\mathcal{V}_p$ arrives after time $t_0$ in a point $p'\in N'$ with speed vector $V'\in \mathcal{V}_{p'}$, then  the geodesic starting at $p$ with speed vector $-V$ arrives after time $t_0$ in $p'$ with speed vector $-V'$, showing that these two geodesics close up to one geodesic. Hence, $M$ equals the image through the exponential map of the compact subset of the normal bundle of $N$ consisting of vectors of norm $\leq t_0$, thus showing that $M\setminus M'=N\cup N'$. 

Consequently, the function $\f_0$ attains its minimum on $N$ and its maximum on $N'$ and has no other critical point. Let $S$ be some level set corresponding to a regular value of $\f_0$. Consider the unit vector field $\zeta:=\frac{I\xi}{|I\xi|}$ on $M'$ (see Figure~\ref{figure} for a visualization of the vector fields $\xi$ and $\zeta$ and of the level sets of $\varphi_0$).

\begin{figure}[h]
 \begin{tikzpicture}

    \draw[thick, color=black] (0,0) ellipse (3cm and 2cm);
    
    \draw[thick, color=black] (-3,0) arc (-180:0:3cm and 1.2cm);
    
    \draw[dashed,thick, color=black] (3,0) arc (0:180:3cm and 1.2cm);

    \draw[thick, color=blue] (0,-2) arc (270:90:0.8cm and 2cm);

  \draw[dashed,thick, color=blue] (0,-2) arc (-90:90:0.8cm and 2cm);

     \draw[->][thick, color=blue] (-4,-0.5) arc (250:120:0.4cm and 1cm);
     
    
   \draw[thick, color=black, domain=-4:-2] plot (\x,{3*\x+9.1});
    
    \draw[thick, color=black, domain=2:4] plot (\x,{3*\x-9.1});
    
    \draw[|-|][thick, color=black] (-4.7,-4.53)--(1.53,-4.5);
   

\draw[->] [thick] (-1,-2.5) -- (-1.5,-4); 

\draw[->] [thick,color=blue] (-0.8,0) -- (-0.8,2.3);

\draw[->] [thick,color=black] (-3,-0.1) -- (-2.3,-1.7); 

\draw[->] [thick,color=black] (0,-1.2) -- (2.2,-1.2);


\draw (-3, -1) node {$V$};    

\draw (-3.3, 0) node {$p$};

\draw (3.5, 0) node {$\gamma(t_0)$};   

\draw (0, -1.5) node {$\gamma(t)$};

\draw (1.5, -1.5) node {$\zeta$};

 \draw (-4.7, -5) node {$\min\varphi_0$};

 \draw (1.3, -5) node {$\max\varphi_0$};

\draw (-0.9, -3.25) node {$\varphi_0$};

\draw [color=blue](-1, 1.5) node {$\xi$};    

\draw [color=black](4.3, 2.5) node {$N'$}; 
\draw [color=black](-1.9, 2.5) node {$N$}; 
\draw [color=blue](-4.6, 0.3) node {$\Phi_s$}; 
\end{tikzpicture}
\caption{}\label{figure}
\end{figure}

From \eqref{der0Jxi} and \eqref{der0xi} we readily compute on $M'$:
\begin{equation}\label{derIxi}
\nabla^0_X\zeta=-\frac{f}{|\xi|}\<X,\xi>\xi,\qquad\forall X\in\T M',
\end{equation}
and
\begin{equation}\label{derzeta}
\nabla^0_\zeta I\zeta=0.
\end{equation}

In particular, we have $\nabla^0_\zeta\zeta=0$, so if $\Psi$ denotes the (local) flow of $\zeta$, the curve $t\mapsto\Psi_t(x)$ is a geodesic for every $x\in M'$, that is, $\Psi_t(x)=\exp_x(t\zeta)$. Note that by \eqref{derIxi}, we have $\dd\zeta^\flat=0$ so the Cartan formula implies $\cL_\zeta\zeta^\flat=\dd(\zeta\lrcorner\zeta^\flat)+\zeta\lrcorner\dd\zeta^\flat=0$, which can also be written as
\begin{equation}\label{zeta}(\cL_\zeta g_0)(\zeta,X)=0,\qquad\forall X\in\T M'.
\end{equation}

We claim that for fixed $t$, $\f_0(\Psi_t(x))$ does not depend on $x\in S$. To see this, let $X\in \T_xS$. By definition $\dd\f_0(X)=0$, whence $g_0(X,\zeta)=0$. We need to show that $\dd\f_0((\Psi_t)_*(X))=0$. This is equivalent to $0=g_0(\zeta,(\Psi_t)_*(X))=(\Psi_t^*g_0)(\zeta,X)$, which clearly holds at $t=0$. Moreover, from \eqref{zeta} we see that the derivative of the function $(\Psi_t^*g_0)(\zeta,X)$ vanishes:
$$\frac{\dd}{\dd t}((\Psi_t^*g_0)(\zeta,X))=(\Psi_t^*\cL_\zeta g_0)(\zeta,X)=(\cL_\zeta g_0)(\zeta,(\Psi_t)_*(X))=0.$$

This shows that for every $x\in S$, $\exp_x(t\zeta)$ belongs to the same level set of $\f_0$. Moreover, $\f_0(\exp_x(t\zeta))$ is decreasing in $t$ since its derivative equals $\dd\f_0(\zeta)=\theta_0(\zeta)=-|\xi|$. Take the smallest $t_1>0$ such that $\pi(x):=\exp_x(t_1\zeta)\in N$ for every $x\in S$.

\noindent {\bf Claim.} The map $\pi$ is a Riemannian submersion from $(S,g_0|_S)$ to $(N,g_0|_N)$ with totally geodesic 1-dimensional fibers tangent to $\xi$. 

\noindent {\it Proof of the Claim.}
First, the Killing vector field $\xi$ commutes with $\zeta$, so $(\Psi_t)_*\xi=\xi$ for all $t<t_1$. Making $t$ tend to $t_1$ implies $\pi_*(\xi_x)=\xi_{\pi(x)}=0$ for every $x\in S$, since $\pi(x)\in N$. Thus $\xi$ is tangent to the fibers of $\pi$.  From \eqref{der0xi} we get $\nabla_{I\zeta} I\zeta=f|\xi|\zeta$, so $I\zeta$ is a geodesic vector field on $S$. Since $I\zeta$ is proportional to $\xi$, it is also tangent to the fibres of $\pi$.

Take now any tangent vector $X\in\T_xS$ orthogonal to $I\zeta$ and denote by $X_t:=(\Psi_t)_*(X)$, which makes sense for all $t<t_1$. By construction we have $\pi_*(X)=\underset{t\to t_1}{\lim}X_t$. Since $0=[\zeta,X_t]=\nabla^0_\zeta X_t-\nabla^0_{X_t}\zeta$, we get by \eqref{derIxi} and \eqref{derzeta}:
$$\zeta(\<X_t,I\zeta>)=\<\nabla^0_\zeta X_t,I\zeta>+\< X_t,\nabla^0_\zeta I\zeta>=\<\nabla^0_{X_t}\zeta,I\zeta>=-f|\xi|\<X_t,I\zeta>.$$
The function $\<X_t,I\zeta>$ vanishes at $t=0$ and satisfies a first order linear ODE along the geodesic $\gamma(t):=\exp_x(t\zeta)$, so it vanishes identically. Thus, $X_t$ is orthogonal to $I\zeta$ for all $t<t_1$. Moreover, the vector field $X_t$ along $\gamma$ has constant norm:
\begin{equation}
\zeta(|X_t|^2)=2\<\nabla^0_\zeta X_t,X_t>=2\<\nabla^0_{X_t}\zeta ,X_t>\stackrel{\eqref{derIxi}}{=} -\frac{2f}{|\xi|}\<X_t,\xi>^2=-2f|\xi|\<X_t,I\zeta>^2=0.
\end{equation}
This shows that $|\pi_*(X)|^2=|X|^2$, thus proving the claim.

Let us now consider the smallest $t_2>0$ such that $\pi(x):=\exp_x(-t_2\zeta)\in N'$ for every $x\in S$ and let $b:=t_1+t_2$.
The flow of the geodesic vector field $\zeta$ defines a diffeomorphism between $(0,b)\times S$ and $M'$, which maps $(r,x)$ onto $\exp_x((r-t_2)\zeta)$. With respect to this diffeomorphism, the vector field $\zeta$ is identified to $\dr$, the metric reads $g_0=\dd r^2+k_r$, where $k_r$ is a family of Riemannian metrics on $S$, and the function $|\theta_0|$ only depends on $r$, say $|\theta_0|=\alpha(r)$. It follows that $\theta_0=\alpha\dd r$ and since $\dd\f_0=\theta_0$, we see that $\f_0=\f_0(r)$ and $\f'_0=\alpha.$

The previous claim actually shows that for every $r\in (0,b)$, 
$k_r=\pi^*(h)+\tau_r\otimes\tau_r,$
where $\tau_r:=I\zeta^\flat$ and $h:=g_0|_N$. From \eqref{derIxi} and \eqref{derzeta} we readily obtain 
\begin{equation*}\dot\tau_r=\cL_\zeta(I\zeta^\flat)=-f\alpha I\zeta^\flat=-f\alpha\tau_r.
\end{equation*} 
This shows that $\tau_r=\ell(r)\omega$ with $\ell(r):=e^{-\int_0^r f(t)\alpha(t)\dd t}$, where $\omega$ denotes the connection $1$-form on the $S^1$-bundle $S\to N$ induced by the Riemannian submersion $\pi$. Finally, the metric on $M'$ reads $g_0=\dd r^2+\pi^*(h)+\ell^2\omega\otimes\omega$, showing that $g_0$ has the form of the metric described in Proposition \ref{prop}.
\end{proof}

\subsection{Proof of Theorem \ref{p2}}
We can now finish the classification of compact manifolds carrying two conformally related non-homothetic K\"ahler metrics. Assume 
that $(g_+,J_+)$ and $(g_-,J_-)$ are K\"ahler structures on a compact manifold $M$ of real dimension $2n\ge4$ with $g_+=e^{2\f}g_-$ for some non-constant function $\f$. Note that $J_+$ is not conjugate to $ J_-$. Indeed, if $J_+$ were equal to $\pm J_-$, then $\Omega_+=\pm e^{2\varphi}\Omega_-$, so $0=\dd\Omega_+=\pm2e^{2\varphi}\dd\varphi\wedge\Omega_-$ would imply $\dd\varphi=0$, so $\varphi$ would be constant.

We introduce the following notation, in order to use the results from Section \ref{s3}: 
\[g:=g_+,\quad I:=J_+,\quad J:=J_-,\quad \Omega^I:=\Omega_+=g(J_+\cdot,\cdot),\quad \Omega^J:=g(J\cdot,\cdot)=e^{2\varphi}\Omega_-.\] 
Then $(M,g,I)$ is K\"ahler, and $(M,g,J)$ is lcK (in fact globally conformally K\"ahler), with Lee form $\theta:=\dd\varphi$. This last statement follows from \eqref{dOmega}, since $\dd\Omega^J=2e^{2\varphi}\dd\varphi\wedge\Omega_-=2\dd\varphi\wedge\Omega^J$.

The first part of  Theorem \ref{thm kaehler} shows that $I$ and $J$ commute, which proves the statment of Theorem~\ref{p2}  for $n=2$. For $n\geq 3$, the proof of the Theorem \ref{thm kaehler} shows moreover, that after replacing $I$ by $-I$ if necessary, one has $I\theta=J\theta$ and $\tr(IJ)=2n-4$. 

Let us now consider the 2-form  $\sigma:=\frac{1}{2}\Omega^I+\frac{1}{2}\Omega^J$, corresponding to the endomorphism $I+J$ of $\T M$ via the metric $g$. We denote again by $M'$ the open set where $\theta$ is non-vanishing.  By \eqref{J}, on $M'$ we have
\begin{equation}\label{sigma}
\sigma=\frac1{|\theta|^2}\theta\wedge I\theta.
\end{equation}
Since $I$ is $\nabla$-parallel (where $\nabla$ is the Levi-Civita connection of $g$), we obtain by \eqref{nablaJ} that  $\nabla_X\sigma=\frac{1}{2}(X\wedge J\theta+JX\wedge\theta)$. Substituting $\Omega^J=2\sigma-\Omega^I$, and using the fact that $\sigma(\theta)=I\theta$ we obtain  the following formula for the covariant derivative of $\sigma$:
\begin{eqnarray*}
\nabla_X\sigma&=&\frac12\nabla_X\Omega^J=\frac{1}{2}(X\wedge J\theta+JX\wedge \theta)\\
&=&\frac{1}{2}(X\wedge (2\sigma-I)\theta+(2\sigma-I)X\wedge\theta)\\
&=&\frac{1}{2}(X\wedge I\theta-IX\wedge\theta)+\sigma(X)\wedge\theta.
\end{eqnarray*}
Since \eqref{sigma} gives $\theta\wedge\sigma=0$, we get $0=X\lrcorner(\theta\wedge\sigma)=\<X, \theta>\sigma-\theta\wedge\sigma(X)$ for every $X\in\T M$. The previous computation thus yields
\begin{equation}\label{deromega}
\nabla_X\sigma=\frac{1}{2}(X\wedge I\theta-IX\wedge\theta)-\<X, \theta>\sigma, \qquad\forall X\in \T M. 
\end{equation}
We consider now the $2$-form
$\widetilde\sigma:=e^{\varphi}\sigma$. By \eqref{deromega}, its covariant differential reads:
\[\nabla_X \widetilde\sigma=\frac{e^\varphi}{2}(X\wedge I\theta-IX\wedge\theta), \quad\forall X\in \T M.\]
Equivalently, this equation can be written as
\begin{equation}\label{tilom}
\nabla_X\widetilde\sigma=\frac{1}{2}(\dd (\tr \,\widetilde\sigma)\wedge IX-\dd^c (\tr\, \widetilde\sigma)\wedge X), \quad\forall X\in \T M,
\end{equation}
where $\tr\, \widetilde\sigma:=\<\widetilde\sigma, \Omega^I>=e^\varphi$ is the trace with respect to the K\"ahler form $\Omega^I$ and $\dd^c$ denotes the twisted exterior differential defined by $\dd^c \alpha:=\sum_iI e_i\wedge \nabla_{e_i}\alpha$, for any form $\alpha$. 

A real $(1,1)$-form on a K\"ahler manifold $(M,g,I, \Omega^I)$ satisfying \eqref{tilom} is called 
a {\it Hamiltonian $2$-form} (see \cite{ACG2006}). Compact K\"ahler manifolds carrying such forms are completely described in \cite[Theorem 5]{ACGT2004}.
In the case where the Hamiltonian form has rank 2, these are exactly the manifolds obtained from the Calabi Ansatz described in Proposition \ref{prop}. 

However, the statement and the proof of \cite[Theorem 5]{ACGT2004} are rather involved, and it is not completely clear that the construction described in the conclusion of \cite[Theorem 5]{ACGT2004} is equivalent to the Calabi Ansatz. We will thus provide here a more direct proof. 

All we need is to show that the globally conformally K\"ahler structure on $M$ determined by $g_0:=e^{\f}g_-=e^{-\f}g_+$ and $I:=J_+$ satisfies the hypotheses of Proposition \ref{thmconv}. We start with the following:

\begin{Lemma}\label{lem nablaTheta}
On the open set $M'$ where $\theta$ is not vanishing, 
the covariant derivative of $\theta$ with respect to $g$ is given by
\begin{equation}\label{et}
\nabla_X\theta=
\frac{1}{2}|\theta|^2 X-\frac{1}{2}\left(\frac{\delta\theta}{|\theta|^2}+n+1\right)\<X,\theta>\theta-
\frac{1}{2}\left(\frac{\delta\theta}{|\theta|^2}+n-1\right)\<X,I\theta>I\theta.
\end{equation}
\end{Lemma}
\begin{proof}
Using the fact that $I$ and $J$ commute, $I\theta=J\theta$ and $\tr(IJ)=2n-4$, \eqref{contr2} simplifies to
\begin{equation}\label{nablath}
\sum_{i=1}^{2n}\<IJ\nabla_{e_i}\theta,e_i>=2(n-1)|\theta|^2+\delta\theta.
\end{equation}
Substituting this into \eqref{contr1},
we obtain
\begin{multline}\label{contr12}
2(2-n)\<Y,\theta> J\theta  
-2n\<Y,J\theta>\theta 
+(2n-5)|\theta|^2 JY+|\theta|^2IY+\delta\theta(JY+ IY)\\
+2(2-n) J\nabla_{Y}\theta +2\nabla_{JY} \theta+2(n-2)\nabla_{IY}\theta-2IJ \nabla_{IY}\theta=0.
\end{multline}
Differentiating \eqref{sigma} on $M'$ yields
\begin{equation}
\nabla_X \sigma=-\frac{2\<\nabla_X
\theta,\theta>}{|\theta|^4}\theta\wedge I\theta+\frac{1}{|\theta|^2}(\nabla_X\theta\wedge I\theta+
\theta\wedge I\nabla_X\theta).
\end{equation}
Comparing with \eqref{deromega}, we obtain
\begin{equation}
\frac{1}{2}(X\wedge I\theta-IX\wedge\theta)-\<X, \theta>\sigma=-\frac{2\<\nabla_X
\theta,\theta>}{|\theta|^4}\theta\wedge I\theta+\frac{1}{|\theta|^2}(\nabla_X\theta\wedge I\theta+
\theta\wedge I\nabla_X\theta).
\end{equation}
Taking the interior product with $I\theta$ in the last equality, we get
\begin{equation}
\frac{1}{2}\<X,I\theta>I\theta-
\frac{1}{2}|\theta|^2X +\frac{1}{2}\<X,\theta>\theta=\frac{\<\nabla_X
\theta,\theta>}{|\theta|^2}\theta+\frac{1}{|\theta|^2}\<\nabla_X\theta,I\theta>
I\theta-\nabla_X\theta.
\end{equation}
We deduce that the following equality holds: 
\begin{equation}\label{eqn nabla}
\nabla_X\theta=\frac{1}{2}|\theta|^2X+\alpha(X)\theta+\beta(X)
I\theta,
\end{equation}
where $\alpha$ and $\beta$ are the following $1$-forms:
\begin{equation}\label{ab}
\alpha=\frac{1}{2}\left(\frac{\dd(|\theta|^2)}{|\theta|^2}-\theta\right),\quad \beta=\frac{1}{|\theta|^2}\nabla_{I\theta}\theta-\frac{1}{2}I\theta.
\end{equation}
Since $\theta$ is closed, \eqref{eqn nabla} yields
$0=\alpha\wedge\theta+
\beta\wedge I\theta.$
Therefore, there exist $a,b,c\in C^\infty(M')$, such that 
$$\alpha=a\theta+b I\theta\quad\textrm{ and }\quad\beta= b\theta+c I\theta.$$ 
Moreover, \eqref{ab} shows that $\alpha$ is closed, so $\dd a\wedge\theta+\dd b\wedge I\theta+b\dd(I\theta)=0$.  On the other hand, by \eqref{eqn nabla}, we have $\dd(I\theta)=|\theta|^2\Omega^I+\alpha\wedge I\theta-\beta\wedge \theta=|\theta|^2\Omega^I+(a+c)\theta\wedge I\theta$. Hence,
\begin{equation*}
\dd a\wedge\theta+\dd b\wedge I\theta+b|\theta|^2\Omega^I+b(a+c)\theta\wedge I\theta=0.
\end{equation*} 
Applying the last equality to
$X$ and $IX$, for a non-zero vector field $X$ orthogonal to $\theta$ and $I\theta$ yields $b=0$. By \eqref{eqn nabla} again we have
\begin{equation}\label{eqn ac1}
-\delta\theta=\sum_{i=1}^{2n}\<e_i,\nabla_{e_i}\theta>=(n+a+c)|\theta|^2.
\end{equation}
Substituting $Y$ by $\theta$ in \eqref{contr12} and using \eqref{eqn nabla}, we obtain
\begin{equation}\label{eqn ac2}
\left(\delta\theta+(1+(2-n)a+nc)|\theta|^2 \right)I\theta=0.
\end{equation}
From \eqref{eqn ac1} and \eqref{eqn ac2}, it follows that 
$$a=-\frac{1}{2}\left(\frac{\delta\theta}{|\theta|^2}+n+1\right)\quad \text{and} \quad c=-\frac{1}{2}\left(\frac{\delta\theta}{|\theta|^2}+n-1\right).$$
This proves the lemma. 
\end{proof}

We write \eqref{et} as
\begin{equation}\label{et1}
\nabla_X\theta=
\frac{1}{2}g(\theta,\theta) X^\flat-\frac12(f+2)\theta(X)\theta-
\frac12fI\theta(X)I\theta,
\end{equation}
where $f:=\left(\frac{\delta\theta}{|\theta|^2}+n-1\right)$. Note that we no longer identify vectors and 1-forms in this relation, since we will now perform a conformal change of the metric.

Namely, we consider the ``average metric" $g_0:=e^{\f}g_-=e^{-\f}g_+$ and denote by $\nabla^0$ its Levi-Civita covariant derivative, by $\theta_0$ the Lee form of $I:=J_+$ with respect to $g_0$ and by $\Omega_0:=g_0(I\cdot,\cdot)$. Since $\dd \Omega_0=\dd(e^{-\f}\Omega_+)=-e^{-\f}\dd \f\wedge\Omega_+=-\dd \f\wedge\Omega_0$, we get $\theta_0=-\frac12\dd \f=-\frac12\theta$.

From \eqref{et1} we immediately get
\begin{equation}\label{et2}
\nabla_X\theta_0=-
g(\theta_0,\theta_0) X^\flat+(f+2)\theta_0(X)\theta_0+
fI\theta_0(X)I\theta_0.
\end{equation}

The classical formula relating the covariant derivatives of $g$ and $g_0=e^{-\f}g$ on 1-forms reads
$$\nabla^0_X\eta=\nabla_X\eta+g(\theta_0,\eta)X^\flat-\eta(X)\theta_0-\theta_0(X)\eta,\qquad\forall X\in \T M,\ \forall \eta\in\Omega^1(M),$$
where $\flat$ is the index lowering with respect to $g$. For $\eta=\theta_0$, \eqref{et2} becomes exactly \eqref{der0theta}.

From the proof of Theorem~\ref{thm kaehler} it is clear that the distribution $\cV:=\ker(I-J)$ is spanned along $M'$ by $\xi$ and $I\xi$, where
$\xi$ denotes the vector field on $M$ corresponding to $I\theta_0$ via the metric $g_0$. This shows that the hypotheses of Proposition~\ref{thmconv} are verified, thus concluding the proof of Theorem \ref{p2}.

\subsection{ Proof of Theorem~\ref{p3}}

Using the above results, we can now complete the classification of compact proper lcK manifolds 
$(M^{2n},g,J,\theta)$ with non-generic holonomy, by reviewing the possible cases in the Berger-Simons holonomy theorem.

First, by Proposition~\ref{prop sym}, there exist no compact irreducible locally symmetric proper lcK manifolds.

In Section~\ref{sec hol}, we showed that if the restricted holonomy of $(M,g)$ is in the Berger list, then necessarily $\mathrm{Hol}_0(M,g)=\mathrm{U}(n)$. 
After passing to a double covering if neccesary, there exists a complex structure $I$, such that $(M,g,I)$ is K\"ahler. By Theorem \ref{thm kaehler}, $I$ and $J$ commute and $(M,g,J,\theta)$ is gcK. 
The conformal class of $g$ ths contains two non-homothetic K\"ahler metrics. We conclude then by Theorem~\ref{p2} that 
$(M,g,J,\theta)$ falls in one of the cases \ref{item 2}a) or \ref{item 2}b) in Theorem~\ref{p3}.

Finally, if $\mathrm{Hol}_0(M,g)$ is reducible, then Theorem~\ref{thmred} shows that $\mathrm{Hol}_0(M,g)$ is (up to conjugation) a subgroup of $\SO(2n-1)$ acting irreducibly on $\RM^{2n-1}$. Theorem~\ref{main am} implies that $(M^{2n},g,J,\theta)$ satisfies either case \ref{item 1} or case \ref{item 2}c) in Theorem~\ref{p3}. Moreover, Corollary \ref{cor-hol} shows that the restricted holonomy group $\mathrm{Hol}_0(M,g)$ is conjugate to $\SO(2n-1)$ in both cases.

\bibliography{bibliographie}

\begin{thebibliography}{10}

\bibitem{ACG2006}
V.~Apostolov, D.~M.~J. Calderbank, and P.~Gauduchon, \emph{Hamiltonian 2-forms
  in {K}\"ahler geometry. {I}. {G}eneral theory}, J. Differential Geom.
  \textbf{73} (2006), no.~3, 359--412.

\bibitem{ACG2013} V. Apostolov, D. M. J. Calderbank, P. Gauduchon, 
\emph{Ambitoric geometry I: Einstein metrics and extremal ambikaehler structures}, 
to appear in J. reine angew. Math.

\bibitem{ACGT2004}
V.~Apostolov, D.~M.~J. Calderbank, P.~Gauduchon, and C.~W.
  T{\o}nnesen-Friedman, \emph{Hamiltonian 2-forms in {K}\"ahler geometry. {II}.
  {G}lobal classification}, J. Differential Geom. \textbf{68} (2004), no.~2,
  277--345.

\bibitem{BM2015}
F.~Belgun and A.~Moroianu, \emph{On the irreducibility of locally metric
  connections}, doi: 10.1515/crelle-2013-0128, to appear in J. reine angew. Math.

\bibitem{Bergery1982} L. B{\'e}rard-Bergery, 
\emph{ Sur de nouvelles vari\'et\'es riemanniennes d'Einstein}. Institut \'Elie Cartan, {\bf6} 
(1982), 1--60.

\bibitem{Besse2008}
A.~L. Besse, \emph{Einstein manifolds}, Classics in Mathematics,
  Springer-Verlag, Berlin, 2008, Reprint of the 1987 edition.

\bibitem{Buchdahl1999} N. Buchdahl, {\sl On compact K\"ahler surfaces},
Ann. Inst. Fourier {\bf 49} no. 1 (1999), 287--302.

\bibitem{Calabi1982}
E.~Calabi, \emph{Extremal {K}\"ahler metrics}, Seminar on {D}ifferential
  {G}eometry, Ann. of Math. Stud., vol. 102, Princeton Univ. Press, Princeton,
  N.J., 1982, 259--290.



\bibitem{CLW2008}
X. Chen, C. LeBrun and B. Weber,
\emph{On conformally {K}\"ahler, {E}instein manifolds}, J. Amer. Math. Soc.
\textbf{21} (2008), no.~4, 1137--1168.

\bibitem{Derdzinski1983}
A.~Derdzinski, \emph{Self-dual {K}\"ahler manifolds and {E}instein manifolds of dimension four},
Compositio Math. \textbf{49} (1983),
no.~3, 405--433.

\bibitem{DM2003}
A.~Derdzinski and G.~Maschler, \emph{Local classification of
  conformally-{E}instein {K}\"ahler metrics in higher dimensions}, Proc. London
  Math. Soc. \textbf{87} (2003), no.~3, 779---819.

\bibitem{DM2005}
\bysame, \emph{A moduli curve for compact conformally-{E}instein {K}\"ahler
  manifolds}, Compositio Math. \textbf{141} (2005), 1029--1080.

\bibitem{DM2006}
\bysame, \emph{Special {K}\"ahler-{R}icci potentials on compact {K}\"ahler
  manifolds}, J. reine angew. Math. \textbf{593} (2006), 73--116.

\bibitem{GMS2011}
P.~Gauduchon, A.~Moroianu, and U.~Semmelmann, \emph{Almost complex structures
  on quaternion-{K}\"ahler mani\-folds and inner symmetric spaces}, Invent. Math.
  \textbf{184} (2011), no.~2, 389--403.

\bibitem{Lamari1999} A. Lamari, {\sl Courants k\"ahl\'eriens et surfaces compactes}, Ann. Inst. Fourier {\bf 49} no. 1 (1999), 263--285.

\bibitem{LeBrun1997}
C. LeBrun, {\sl Einstein metrics on complex surfaces}, Geometry and physics (Aarhus, 1995), Lecture Notes in Pure and Appl. Math., vol. 184, Dekker, New York, 1997, 167--176.


\bibitem{Lichnerowicz1957}
A. Lichnerowicz, \emph{Sur les transformations analytiques des
  vari\'et\'es k\"ahl\'eriennes compactes}, C. R. Acad. Sci. Paris \textbf{244}
  (1957), 3011--3013.

\bibitem{Moroianu2007B}
A.~Moroianu, \emph{Lectures on {K}\"ahler geometry}, London Mathematical
  Society Student Texts, vol.~69, Cambridge University Press, Cambridge, 2007.

\bibitem{Moroianu2015}
\bysame, \emph{Compact lc{K} manifolds with parallel vector fields}, Complex
  Manifolds \textbf{2} (2015), 26--33.

\bibitem{Oneill1983}
B.~O'Neill, \emph{Semi-{R}iemannian geometry}, Pure and Applied Mathematics,
  vol. 103, Academic Press, Inc., New
  York, 1983.


\bibitem{Page1978} D. Page, \emph{A compact rotating gravitational instanton}, Physics Letters B 
{\bf 79} (1978), no.~3, 235--238.

\bibitem{Tashiro1962}
Yoshihiro Tashiro, \emph{On conformal and projective transformations in
  {K}\"ahlerian manifolds}, T\^ohoku Math. J. \textbf{14} (1962), no.~2,
  317--320.

\bibitem{Vaisman1980}
I.~Vaisman, \emph{On locally and globally conformal {K}\"ahler manifolds},
  Trans. Amer. Math. Soc. \textbf{262} (1980), no.~2, 533--542.

\end{thebibliography}

\end{document}